\documentclass[11pt]{amsart}
\usepackage{graphicx}
\usepackage{amssymb}
\usepackage{mathrsfs}
\usepackage{epstopdf}
\usepackage{hyperref}
\usepackage{color}
\usepackage{pgfplots}

\theoremstyle{definition} 
\theoremstyle{plain}      
\theoremstyle{plain}      
\theoremstyle{plain}      
\theoremstyle{plain}      \newtheorem{claim}{Claim}
\theoremstyle{plain}      \newtheorem{proposition}{Proposition}
\theoremstyle{plain}      \newtheorem{theorem}{Theorem}
\theoremstyle{remark}     \newtheorem{remark}{Remark}
\theoremstyle{remark}     
\theoremstyle{plain}

\title[Dynamical regimes in a trapped random walk]{Dynamical regimes and finite time behavior in a trapped random walk: a direct iterative approach}
\author{E. Floriani, R. Lima and E. Ugalde}
\address{Aix Marseille Univ, Universit\'e de Toulon, 
CNRS, CPT, Marseille, France}
\address{Dream and Science Factory, Marseille France.}
\address{Instituto de F\'isica, Universidad Aut\'onoma de San 
Luis Potos\'i, M\'exico.}

\email{floriani@cpt.univ-mrs.fr}
\email{dream.and.science@gmail.com}
\email{ugalde@ifisica.uaslp.mx}

\date{}                                         

\begin{document}

\begin{abstract}
We consider a basic one-dimensional model of diffusion which allows to obtain 
a diversity of diffusive regimes whose speed depends on the moments of the 
per-site trapping time. This model is closely related to the continuous time 
random walks widely studied in the literature. The model we consider lends itself to
a detailed treatment, making it possible to study deviations from normality 
due to finite time effects.
\end{abstract}

\maketitle

\bigskip

\section{Introduction}

\subsection{} The random walks with trapping times we study in this 
paper can be seen as the natural discrete version of the continuous 
time random walk (CTRW), which has been treated in mathematical physics 
literature since its introduction by Montroll and Weiss in 
1965~\cite{1965Montroll&Weiss}. A recent account on the subject can 
be found in~\cite{2014Metzler&al}, where CTRWs are considered as models 
for anomalous diffusion. A presentation of the relation between CTRWs 
and fractional diffusion can be found in~\cite{2000Metzler&Klafer}. 
The asymptotic behavior of continuous time random walks, and consequently 
of the model considered here, is well known and can be found 
in~\cite{2000Barkai&al,1995Kotulski}. In recent years two kind of 
generalizations of the continuous time random walk have been studied: 
on the one hand there are models including spatial 
inhomogeneities~\cite{2002Fontes&al,2015BenArous&al}, otherwise 
known as random environment; on the other hand, some models consider 
correlations between the space and time variables~\cite{2004Becker-Kern&al}. 
In both directions limit theorems have been obtained, and the relation 
to fractional dynamics and anomalous diffusion has been exposed. 
The aim of the present work is quite modest in comparison. We will 
treat in full detail, and from elementary grounds, the one-dimensional 
and discrete time version of the CTRW. We will be concerned in 
particular with the speed of convergence towards normal diffusion, 
and with the finite time deviations from normality which can be 
observed in the case of trapping time with infinite variance. Our main 
result, then, concerns the deviations from normality, and the emergence 
of a sub-diffusive behavior due to finite time effects, in the case of 
infinite variance. In the case of infinite mean trapping time, we study 
among other phenomena, the scaling behavior of the random walk and the 
deviations from this asymptotic scaling law at finite times. 

\subsection{} In our model, a particle performs a random walk in a 
one-dimensional lattice in such a way that at each site it can stay 
trapped for an integer random time. The properties of the diffusion 
process depend on the distribution of this trapping time. As mentioned 
above, the model can be seen as a natural discrete version of the 
CTRW widely studied in the literature. We will recover, from elementary 
grounds, a classical limit theorem, as for instance the fact that 
for trapping times with finite mean, the particle performs a random 
walk consistent with a diffusive process with a diffusion coefficient 
decreasing with the expected trapping time. In the case of infinite 
mean trapping time, a sub-diffusive behavior is obtained. In both 
cases we are able to determine the squared mean displacement at all 
times, and to evaluate its asymptotic behavior. Our formulas give 
quite precise estimates of the deviations from the asymptotic scaling 
at finite times. To the best of our knowledge, these are the first 
results of this kind. We are also able, in the case of finite mean 
trapping times, to estimate the speed of convergence towards normal 
diffusion, by using elementary probability inequalities. In this way 
we develop a direct approach to diffusion regimes leading to explicit 
estimates of rates of convergence and finite time deviations. 

\subsection{} The paper is organized as follows. In the next section we 
define the model and present some basic general results. In particular, 
we establish a formula that allows us to compute the evolution of the mean 
squared displacement (MSD) of the random walk. In Sections~\ref{sec:diffusive} 
and~\ref{sec:subdiffusive} we respectively examine the diffusive and 
sub-diffusive regimes, paying special attention to the finite time 
behavior of the MSD. Finally, in the last section we summarize  
our results and present concluding remarks.

\bigskip

\section{Generalities}\label{sec:generalities}  
\subsection{The model as a Markov chain}
\label{subsec:markov}
At each site $z\in {\mathbb Z}$, the diffusive particle gets trapped 
a random time $T_z\in {\mathbb N}_0$, after which it moves with equal 
probability to either of the two neighboring sites $z-1$ or $z+1$. 
We will consider the homogeneous case, for which the distribution of 
the random trapping time, $p_z(\tau):={\mathbb P}(T_z=\tau)$, does not 
depend on the site $z\in{\mathbb Z}$. Hence, we model the process as 
a discrete-time Markov chain on the set $E={\mathbb Z}\times {\mathbb N}_0$, 
where the first component represents the position of the random walker 
and the second one the trapping time. The transition probabilities are
\begin{equation}\label{eq:markovchain}
{\mathbb P}((X,T)_{t+1}=(z',\tau')|\,(X,T)_t
=(z,\tau))=
\begin{cases} 
p(\tau')/2& \mbox{ if } |z-z'|=1 
            \mbox{ and } \tau=0 \, ,            \\
        1 & \mbox{ if } z=z' 
            \mbox{ and } \tau'=\tau-1 \,\geq \, 0 \,,\\
        0 & \mbox{ otherwise.}
\end{cases} 
\end{equation} 
We are interested in the statistical behavior of the position $X_t$, as 
a function of the distribution $\{p(\tau):\tau\in{\mathbb N}_0\}$. We will 
assume that the particle starts its random walk at the position $z=0$, so 
that 
\begin{equation}\label{eq:initialcondition}
{\mathbb P}((X,T)_{t}=(z,\tau))=
\begin{cases} 
             p(\tau) & \mbox{ if } z=0,\\
             0       & \mbox{ otherwise,}
\end{cases}             
\end{equation}
for all $t\leq 0$.

\medskip\noindent Notice that the restriction of the process to the 
variable $T$ is also a Markov chain:
\begin{equation}\label{eq:markovchainT}
{\mathbb P}(T_{t+1}=\tau'|\,T_t=\tau)=
\begin{cases} 
p(\tau')& \mbox{ if } \tau=0 \, ,            \\
        1 & \mbox{ if } \tau'=\tau-1 \,\geq \, 0 \,,\\
        0 & \mbox{ otherwise.}
\end{cases} 
\end{equation} 

\begin{center}
\begin{figure}[h]

\begin{tikzpicture}[scale=0.7]
\draw[help lines] (1,0) grid (11,7);
\draw[<->] (0,0) -- (12,0);
\draw[->] (6,0) -- (6,8);
\draw[blue, ->] (6.1,0) -- (6.9,6);
\draw[blue, ->] (7.05,6) -- (7.05,5);
\draw[blue, ->] (7.05,5) -- (7.05,4);
\draw[blue, ->] (7.05,4) -- (7.05,3);
\draw[blue, ->] (7.05,3) -- (7.05,2);
\draw[blue, ->] (7.05,2) -- (7.05,1);
\draw[blue, ->] (7.05,1) -- (7.05,0.1); 
\draw[blue, ->] (7.1,0) -- (7.9,4);
\draw[blue, ->] (8.05,4) -- (8.05,3);
\draw[blue, ->] (8.05,3) -- (8.05,2);
\draw[blue, ->] (8.05,2) -- (8.05,1);
\draw[blue, ->] (8.05,1) -- (8.05,0.1);   
\draw[blue, ->] (8.1,0) -- (8.9,7);
\draw[blue, ->] (9.05,7) -- (9.05,6);
\draw[blue, ->] (9.05,6) -- (9.05,5);
\draw[blue, ->] (9.05,5) -- (9.05,4);
\draw[blue, ->] (9.05,4) -- (9.05,3);
\draw[blue, ->] (9.05,3) -- (9.05,2);
\draw[blue, ->] (9.05,2) -- (9.05,1);
\draw[blue, ->] (9.05,1) -- (9.05,0.1);  
\draw[red, ->] (8.9,0)-- (8.1,6);
\draw[red, ->] (7.95,6) -- (7.95,5);
\draw[red, ->] (7.95,5) -- (7.95,4);
\draw[red, ->] (7.95,4) -- (7.95,3);
\draw[red, ->] (7.95,3) -- (7.95,2);
\draw[red, ->] (7.95,2) -- (7.95,1);
\draw[red, ->] (7.95,1) -- (7.95,0.1);
\draw[red, ->] (7.9,0) -- (7.1,5);
\draw[red, ->] (6.95,5) -- (6.95,4);
\draw[red, ->] (6.95,4) -- (6.95,3);
\draw[red, ->] (6.95,3) -- (6.95,2);
\draw[red, ->] (6.95,2) -- (6.95,1);
\draw[red, ->] (6.95,1) -- (6.95,0.1);
\draw[red, ->] (6.9,0) -- (6.1,1);
\draw[red, ->] (5.95,1) -- (5.95,0.1);
\draw[red, ->] (5.9,0) -- (5.1,6);
\draw[red, ->] (4.95,6) -- (4.95,5);
\draw[red, ->] (4.95,5) -- (4.95,4);
\draw[red, ->] (4.95,4) -- (4.95,3);
\draw[red, ->] (4.95,3) -- (4.95,2);
\draw[red, ->] (4.95,2) -- (4.95,1);
\draw[red, ->] (4.95,1) -- (4.95,0.05);
\draw[blue, ->] (5.1,0) -- (5.9,4);
\node [right] at (6,7.5) {$T$};  
\node [below] at (11.5,0) {$X$}; 
\draw[blue] (6,0) circle [radius=0.1];   
\node [below] at (6,0) {$t=0$};
\draw[red] (6,4) circle [radius=0.1];   
\node [above] at (6,4) {$t=43$};
\end{tikzpicture}
\caption{The trajectory in the $(X,T)$ plane of a possible realization 
of the process. The trajectory starts at $(X,T)=(0,0)$ and reaches 
$(X,T)=(0,4)$ at time $t=43$. The parts of the trajectory 
where $X$ increases are in blue, the parts where $X$ decreases 
are in red.}~\label{fig:realizationA}
\end{figure}
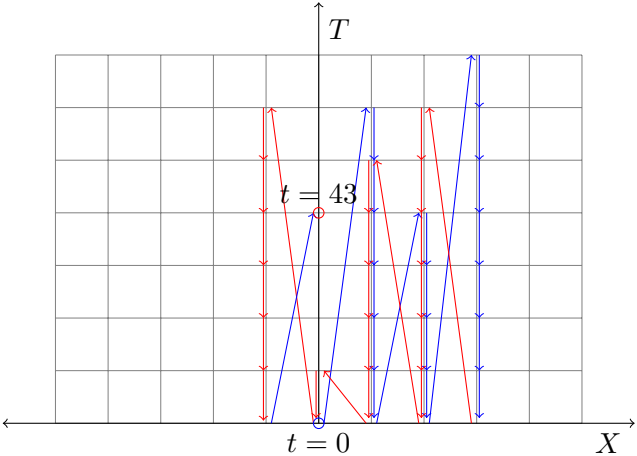 
\end{center}

\begin{center}
\begin{figure}[h]
\begin{tikzpicture}
\definecolor{mybarcolor}{RGB}{210 207 155}
\pgfmathsetmacro{\xmin}{-4}
\pgfmathsetmacro{\ymin}{-4}
\begin{axis}[ymin=0,%
  xmin=\xmin,%
  ymin=\ymin,
  ylabel=\textbf{\hskip 60pt $X_t$\hskip 10pt $T_t$},%
  xlabel=\textbf{\hskip 80pt  $t$},%
  grid=major,%
]
\addplot[ color = blue] table{%
0 0 
1 1
2 1
3 1
4 1
5 1
6 1 
7 1
8 2
9 2
10 2
11 2
12 2
13 3
14 3
15 3
16 3
17 3
18 3 
19 3
20 3
21 2
22 2
23 2
24 2
25 2
26 2
27 2
28 1
29 1
30 1
31 1
32 1
33 1
34 0
35 0
36 -1
37 -1
38 -1
39 -1
40 -1
41 -1
42 -1
43 0
};

\addplot[ color = blue, dashed] table{%
0 0 
1 6
2 5
3 4
4 3
5 2
6 1 
7 0
8 4
9 3
10 2
11 1
12 0
13 7
14 6
15 5
16 4
17 3
18 2 
19 1
20 0
21 6
22 5
23 4
24 3
25 2
26 1
27 0
28 5
29 4
30 3
31 2
32 1
33 0
34 1
35 0
36 6
37 5
38 4
39 3
40 2
41 1
42 0
43 4
};
\end{axis}
\end{tikzpicture}

\caption{The sequence of positions $X_t$ (solid line) and trapping 
times $T_t$ (dashed line), corresponding to the trajectory 
depicted in Figure~\ref{fig:realizationA}.}
\label{fig:realizationB}
\end{figure}
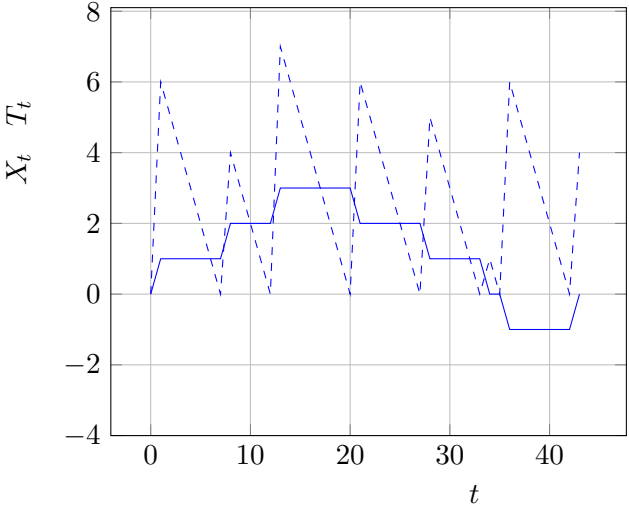
\end{center}

\medskip\noindent 
\subsection{The model as a subordinated random walk}
\label{subsec:subordinated} 
We can write the trapped random walk as a standard binary random walk 
subordinated to a monotonously increasing random variable governing 
the jump times. For this note that the space increments 
\[
\Delta_t:=X_{t+1}-X_t\in \{-1,0,1\},
\]
form a sequence of random variables satisfying 
$\mathbb{P}(\Delta_t=1\,|T_t=0)
          =\mathbb{P}(\Delta_t=-1\,|T_t=0)=1/2$, 
and $\mathbb{P}(\Delta_t=0\,|T_t>0)=1$.
Here $T_t$ denotes the trapping at time $t$ of the random waker. A 
direct computation shows that for each finite collection of times 
$t_1<t_2<\cdots <t_n$ and corresponding increments 
$\epsilon_1,\epsilon_2,\ldots, \epsilon_n\in \{-1,1\}$, we have
\[ 
\mathbb{P}(\Delta_{t_1}=\epsilon_1,\ldots,\Delta_{t_n}
            =\epsilon_k\,| T_{t_1}=\cdots=T_{t_n}=0)
            =\prod_{k=1}^n\mathbb{P}(
            \Delta_{t_k}=\epsilon_k\, |T_{t_k}=0)
=\frac{1}{2^n}.
\]
Therefore the increments $\Delta_t$ are independent when conditioned 
to the escape level $E_0:=\{T=0\}$ which is nothing but a renewal time. 
From this it readily follows that
\[
X_{t+1}=X_t+\chi_{\{T_t=0\}}\, \Delta_t,
\]
where $\chi_{\{T_t=0\}}$ is the characteristic function for the event 
$\{T_t=0\}\subset E_0$, and $\Delta_t$ is a sequence of i.i.d. random 
variables such that $\mathbb{P}(\Delta_t=1)=\mathbb{P}(\Delta_t=-1)=1/2$. 
From this we obtain
\[
X_{t+1}=\sum_{s=0}^{t} \chi_{\{T_s=0\} }\,
\Delta_s=\sum_{n=0}^{N_t}\Delta_n,
\]
where $N_t:=\#\{0\leq s \leq t:\ T_s=0\}$ is nothing but the number of 
times the random walker visited the escape level up to time $t$. Since 
$\Delta_n$ and $N_t$ are independent, then 
\begin{align}\label{eq:subordinatedrw}
\mathbb{P}\left(X_t=z\right)
&=\sum_{n=0}^{t-1} \mathbb{P}(N_{t-1}=n)\,
  \mathbb{P}\left( \sum_{m=0}^n \Delta_m=z\right)
\nonumber\\
&=\sum_{n=0}^{t-1}2^{-n}\mathbb{P}(N_{t-1}=n)\
  \left(\begin{array}{c} n \\                                                                     
              \frac{n+z}{2}\end{array}\right).
\end{align}  

\medskip \noindent 
\subsection{A recurrence for the MSD}
\label{subsec:sigmarecurrence} 
In order to quantify the diffusivity of the dynamics, we focus now on 
the mean squared displacement (MSD)
\begin{equation}\label{eq:MSD}
\sigma_t^2:={\mathbb E}(X_t^2)=
      \sum_{z\in{\mathbb Z}}z^2{\mathbb P}(X_t=z).
\end{equation} 
We have the following.

\medskip
\begin{proposition}[Recurrence relation]
\label{prop:sigmarecurrence}
Suppose that the trapped random walk $\{(X,T)_t\}_{t\in\mathbb{Z}}$ is 
at the origin at time $t=0$, i.e., it satisfies the initial 
condition~\eqref{eq:initialcondition}. Then, for each $t\in\mathbb{N}$ 
we have,
\[
\sigma_{t+1}^2=\sum_{0\leq\tau\leq t}p(\tau) 
                 \left(\sigma^2_{t-\tau}+1\right).
\]
\end{proposition}

\medskip\noindent
We will use this recurrence relation to quantify the diffusivity of the 
trapped random walk as a function of the trapping time distribution. 

\medskip
\begin{proof}
For each $t\geq 1$ we have
\begin{align*}
{\mathbb P}(X_t=z)
&=\sum_{0\leq \tau} {\mathbb P}((X,T)_t=(z,\tau))
 =\sum_{0\leq \tau}\sum_{0\leq \tau'\leq t-1}
  \frac{p(\tau+\tau')}{2}
  {\mathbb P}\{(X,T)_{t-\tau'-1}= (z\pm 1,0)\}\\
&=\sum_{0\leq \tau' \leq t-1}{\mathbb P}
  \{(X,T)_{t-\tau'-1}=(z\pm 1,0)\}
        \sum_{0\leq \tau}\frac{p(\tau+\tau')}{2}\\
&=\sum_{0\leq \tau' \leq t-1}
  \frac{{\mathbb P}\{(X,T)_{t-\tau'-1}=(z\pm 1,0)\}}{2} 
                            \mathbb{P}(T\geq \tau').
\end{align*}
where 
${\mathbb P}\{(X,T)_{t-\tau'-1}= (z\pm 1,0)\}=
       {\mathbb P}\{(X,T)_{t-\tau'-1}= (z-1,0)\}+
       {\mathbb P}\{(X,T)_{t-\tau'-1}= (z-1,0)\}$. 
Hence, the distribution of $X_t$ is determined by its distribution 
restricted to the escape level $E_0$ and the queue of the trapping 
time distribution $\tau\mapsto \mathbb{P}(T\geq\tau)$. The distribution 
of $X_t$ restricted to $E_0$ satisfies the recurrence relation,
\[ 
{\mathbb P}\{(X,T)_{t}= (z,0)\}=\sum_{0\leq \tau\leq t-1}
\frac{p(\tau)}{2}{\mathbb P}\{(X,T)_{t-\tau-1}=(z\pm 1,0)\}.
\]
Using this we obtained
\begin{align*}
\sigma_t^2
&=\sum_{z\in{\mathbb Z}}z^2\sum_{0\leq \tau' \leq t-1}
  \frac{{\mathbb P}\{(X,T)_{t-\tau'-1}= (z\pm 1,0)\}}{2}
  \mathbb{P}(T\geq \tau')\\
&=\sum_{z\in{\mathbb Z}}z^2\sum_{0\leq \tau' \leq t-1}
  \mathbb{P}(T\geq \tau')
  \sum_{0\leq\tau<t-\tau'-1}\frac{p(\tau)}{2}
  \frac{{\mathbb P}\{(X,T)_{t-\tau'-\tau-2}= 
                                     (z-1\pm 1,0)\}}{2}\\
& \hskip 20pt +
\sum_{z\in{\mathbb Z}}z^2\sum_{0\leq \tau' \leq t-1}
  \mathbb{P}(T\geq \tau')
\sum_{0\leq\tau<t-\tau'-1}\frac{p(\tau)}{2}
\frac{{\mathbb P}\{(X,T)_{t-\tau'-\tau-2}= 
                                    (z+1\pm 1,0)\}}{2},
\end{align*}
which can be rewritten as
\begin{align*}
\sigma_t^2
&=\sum_{0\leq\tau\leq t-1}\frac{p(\tau)}{2}
\sum_{0\leq \tau' \leq t-\tau-2}\mathbb{P}(T\geq \tau')
\sum_{z\in{\mathbb Z}}(z+1)^2
 \frac{{\mathbb P}\{(X,T)_{(t-\tau-1)-\tau'-1}=
                                     (z\pm 1,0)\}}{2}\\
& \hskip 10pt + \sum_{0\leq\tau\leq t-1}\frac{p(\tau)}{2}
       \sum_{0\leq \tau' \leq t-\tau-2}
                  \mathbb{P}(T\geq \tau')
                       \sum_{z\in{\mathbb Z}}(z-1)^2
 \frac{{\mathbb P}\{(X,T)_{(t-\tau-1)-\tau'-1}= 
                                     (z\pm 1,0)\}}{2}\\
&=\sum_{0\leq\tau\leq t-1}\frac{p(\tau)}{2}
   \left(\sigma^2_{t-\tau-1} + 
       2{\mathbb E}\left(X_{t-\tau-1}\right)+1\right)+ 
   \sum_{0\leq\tau\leq t-1}\frac{p(\tau)}{2}
        \left(\sigma^2_{t-\tau-1} - 
       2{\mathbb E}\left(X_{t-\tau-1}\right)+1\right)\\
&=\sum_{0\leq\tau \leq t-1}p(\tau)
    \left(\sigma^2_{t-\tau-1}+1\right).
\end{align*}
\end{proof}

\medskip \noindent 
The recurrence relation just obtained already ensures the unboundedness 
of the trapped random walk. Indeed, we have the following.

\medskip
\begin{proposition}[Unboundedness]
\label{prop:unboundedness}
Suppose that the trapped random walk $\{(X,T)_t\}_{t\in\mathbb{Z}}$ 
satisfies the initial condition~\eqref{eq:initialcondition}. Then the 
sequence $\{\sigma_t^2\}_{t\in{\mathbb N}}$ is increasing and unbounded.
\end{proposition}

\medskip 
\begin{proof}
For the first claim it is enough to notice that for each $t\geq 0$, 
$\Delta\sigma_{t+1}^2:=\sigma_{t+1}^2-\sigma_t^2$ satisfies the 
recurrence relation 
$\Delta\sigma_{t+1}^2=
      p(t)+\sum_{0\leq \tau\leq t}p(\tau)\Delta\sigma^2_{t-\tau}$, 
and since $\Delta\sigma_1^2=p(0)\geq 0$, it follows by induction on $t$ 
that $\Delta\sigma_{t}^2\geq 0$ for all $t\in {\mathbb N}$. 

\medskip \noindent For the other claim, if we suppose on the contrary 
that the sequence $\{\sigma_t^2\}_{t\in {\mathbb N}}$ is bounded, since 
it is monotonous, then it necessarily has a limit 
$\sigma^2_\infty:=\lim_{t\to\infty}\sigma_t^2$. This limit must satisfy 
\[
\sigma_{\infty}^2=\lim_{t\to\infty}\sum_{0\leq\tau\leq t}p(\tau) 
                          \left(\sigma^2_{t-\tau}+1\right),
\]
but this is impossible since there exist $t_0\in{\mathbb N}$ such that 
$\sigma^2_t\geq \sigma^2_\infty - 1/2$ for each $t\geq t_0$, and therefore 
\[
\lim_{t\to\infty}\sum_{0\leq\tau\leq t}p(\tau) 
   \left(\sigma^2_{t-\tau} +1\right)
         \geq \left(\sigma_\infty^2+1/2\right)         
         \left(\lim_{t\to\infty}
\sum_{0\leq\tau \leq t-t_0}p(\tau)\right) >\sigma_\infty^2.
\]
\end{proof}
\bigskip
\section{Diffusive regime}\label{sec:diffusive}
\noindent In this section we study the situation 
$\mathbb{E}(T^\alpha)<\infty$, for some $\alpha \geq 1$. In this case, 
$X_t$ follows asymptotically a normal diffusion, and satisfies a 
central limit theorem. Nevertheless, the speed of convergence towards 
the normal behavior strongly depends on the trapping time distribution 
tail, and finite time deviations from normality appear. We will analyze 
the convergence towards normal diffusion, by first considering the 
behavior of the mean squared displacement, and then we will prove a 
Central Limit Theorem. 

\subsection{Normal diffusion via MSD}
\label{subsec:NormalMSD}
In all cases when $\mathbb{E}(T)<\infty$, the mean squared 
displacement asymptotically follows a linear growth, with slope, or 
diffusion coefficient, $\mathcal{D}:=(\mathbb{E}(T)+1)^{-1}$. Nevertheless, 
the finite time deviations from this linear growth may give place to 
an apparent sub-linear growth. Our main result concerning this is the 
following.

\medskip
\begin{theorem}[Normal diffusion]\label{teo:NormalMSD}
Let us suppose that $\mathbb{E}(T)<\infty$ and let 
$\mathcal{D}:=(\mathbb{E}(T)+1)^{-1}$. Then, the sequence 
$|\sigma_t^2-\mathcal{D}\,t|$ is bounded if and only if 
$\mathbb{E}(T^2)<\infty$. Furthermore, 
$\lim_{t\to\infty} \mathcal{D} \, t/\sigma_t^2=1$. 
\end{theorem}

\begin{proof}
Let  
$R_t:=\mathcal{D}
      \sum_{\tau\leq t}
      \sum_{\tau'>\tau}(\tau'-\tau)p(\tau')$ 
for each $t\geq 0$. We start by proving  
\begin{equation}\label{eq:MSDinequality}
e^{-q\mathbb{E}(T)}R_t-\mathbb{E}(T)
        \leq \sigma_t^2-\mathcal{D}\,t\leq R_t,
\end{equation}
with $q=-\log(p(0))/(1-p(0))$. For this let 
$\epsilon_t=\sigma_t^2-\mathcal{D}\, t$. According to Proposition~\ref{prop:sigmarecurrence} we have
\begin{align}
\epsilon_{t+1}
&=\sum_{0\leq \tau \leq t} p(\tau)\epsilon_{t-\tau} 
    +\mathbb{P}(T\leq t)
    -\mathcal{D}\, \left(t\mathbb{P}(T>t)+1+
  \sum_{0\leq \tau \leq t}\tau p(\tau)\right),\nonumber\\
&=\sum_{0\leq \tau \leq t} p(\tau)\epsilon_{t-\tau}+   
    \mathcal{D}\,
    \sum_{\tau > t} (\tau-t) p(\tau)-\mathbb{P}(T > t), 
    \label{eq:epsilon}
\end{align} 
which can be written as
\begin{equation}\label{eq:epsilon-delta}
\epsilon_{t+1}
   =\sum_{0\leq \tau \leq t}\delta_\tau^+\, Q_{t-\tau}-
    \sum_{0\leq \tau \leq t}\delta_\tau^-\, Q_{t-\tau},
\end{equation}
with 
$\delta_\tau^+:=\mathcal{D}\,\sum_{\tau > t} (\tau-t) p(\tau)$, 
$\delta_\tau^-:=\mathbb{P}(T > t)$ and $Q_{\tau}$ defined recursively 
by the convolution $Q_{t+1} =\sum_{0\leq\tau\leq t}Q_{t-\tau}\,p(\tau)$, 
starting with $Q_0=1$. Since
\[
0<\mathbb{P}(T\leq t) \min_{\tau\leq t}Q_\tau
  \leq Q_{t+1}
  \leq \mathbb{P}(T\leq t)\max_{\tau\leq t}Q_\tau \leq 1,
\]
then
\begin{equation}\label{eq:boundQ}
e^{-q\, \mathbb{E}(T)}=
  e^{-\alpha\sum_{\tau\geq 0}\mathbb{P}(T>\tau)}
\leq 
\prod_{\tau\geq 0} \mathbb{P}(T\leq \tau) \leq Q_t\leq 1.
\end{equation}
Finally, since 
$0\leq \sum_{0\leq \tau\leq t}\delta_\tau^-
 =\sum_{0\leq \tau\leq t}\mathbb{P}(T < \tau)
\leq \mathbb{E}(T)$, 
taking into account Equations~\eqref{eq:epsilon-delta} and~\eqref{eq:boundQ}, 
we obtain
$e^{-q\, \mathbb{E}(T)}
    \left(\sum_{0\leq \tau\leq t}\delta_\tau^+\right)
    -\mathbb{E}(T) \leq \epsilon_{t+1} 
\leq \sum_{0\leq \tau\leq t}\delta_\tau^+$,
which is exactly~\eqref{eq:MSDinequality} since $\sum_{0\leq\tau\leq t}\delta_\tau^+=R_t$.

\medskip \noindent
To complete the proof, notice that
\[R_t:=\mathcal{D}\,
 \sum_{\tau\leq t}\sum_{\tau'>\tau}(\tau'-\tau)p(\tau')
 =\mathcal{D}\,\sum_{\tau\leq t+1} p(\tau)\,
 \frac{\tau(\tau+1)}{2} +(t+1)\,\mathcal{D}\,
 \sum_{\tau  > t+1}\left(\tau-\frac{t}{2}\right)\,p(\tau).
\]
Hence, when $\mathbb{E}(T^2)<\infty$, we necessarily have 
$t\, \sum_{\tau\geq t}\tau\, p(\tau)\rightarrow 0$ when $t\to \infty$, 
and therefore
$\lim_{t\to \infty}R_t
=\mathcal{D}\left(\mathbb{E}(T^2)+\mathbb{E}(T)\right)/2$,
which implies that $|\sigma_t^2-\mathcal{D}\, t|$ is bounded. If on the 
contrary $\mathbb{E}(T^2)=\infty$, we necessarily have
\[
\liminf_{t\to\infty}\sigma_t^2-\mathcal{D}\,t
\geq 
   \liminf_{t\to\infty}\mathcal{D}\, e^{-q\,\mathbb{E}(T)}\,
   \sum_{\tau\leq t+1} p(\tau)\,\frac{\tau(\tau+1)}{2}-\mathbb{E}(T)
=\infty
\]
In any case, the Inequalities~\eqref{eq:MSDinequality} imply
\[e^{-q\,\mathbb{E}(T)}\limsup_{t\to\infty}R_t/t
\leq \liminf_{t\to\infty} (\sigma_t^2/t-\mathcal{D})
\leq \limsup_{t\to\infty} (\sigma_t^2/t-\mathcal{D})
\leq \liminf R_t/t.\]
Hence, $\lim_{t\to\infty}R_t/t$ exists, and since $\mathbb{E}(T)<\infty$, 
then we necessarily have
\begin{align*}
0\leq \lim_{t\to\infty}\frac{R_t}{t}
&\leq \frac{1}{2}\lim_{t\to\infty}\frac{1}{t}
  \sum_{\tau\leq t+1}\tau^2\,p(\tau)\\
& = \frac{1}{2}\lim_{t\to\infty}\frac{1}{t}
\sum_{\tau\leq \sqrt{t}}\tau^2\,p(\tau)+
\frac{1}{2}\lim_{t\to\infty}\frac{1}{t}
\sum_{\sqrt{t}<\tau\leq t}\tau^2\,p(\tau)\\
&\leq \frac{1}{2}\lim_{t\to\infty}
      \frac{1}{\sqrt{t}}\mathbb{E}(T)+
\frac{1}{2}\lim_{t\to\infty}
\sum_{\sqrt{t}<\tau \infty}\tau\,p(\tau)=0
\end{align*}
and therefore $\lim_{t\to\infty}\sigma_t^2/t=\mathcal{D}$.
\end{proof}

\bigskip\noindent
The diffusive case is the one where the $T$-Markov 
chain~\eqref{eq:markovchainT} admits a stationary distribution 
$\pi(\tau)$. Indeed, the stationarity condition is 
\[  \pi(\tau') = \sum_{\tau=0}^{\infty} \pi(\tau) 
{\mathbb P}(T_{t+1}=\tau'|\,T_t=\tau) 
= \pi(0) p(\tau') + \pi(\tau' + 1) \]
so that
\[  \pi(\tau) = \pi(0) \sum_{k=\tau}^{\infty} p(k) \]
The normalization condition
\[ 1 = \sum_{\tau=0}^{\infty} \pi(\tau) = 
\pi(0)  \sum_{\tau=0}^{\infty} \sum_{k=\tau}^{\infty} p(k) =  
\pi(0)  \sum_{\tau=0}^{\infty} (\tau+1) p(\tau) 
= \pi(0) (\mathbb{E}(T)+1) \]
implies that the stationary distribution exists if and only if 
$\mathbb{E}(T)$ is finite. In this case
\[  \pi(\tau) = \frac{1}{\mathbb{E}(T)+1} \sum_{k=\tau}^{\infty} 
p(k) = \frac{1}{\mathcal{D}} \sum_{k=\tau}^{\infty} p(k). \]
\bigskip

\paragraph{\emph{Exponential trapping time}.} 
A nice example is supplied by the exponential trapping time 
distribution. Let us assume that $p(\tau)=(1-\lambda)\lambda^{\tau}$ 
for each $\tau$. In this case 
\[{\mathbb E}(T)+1=
(1-\lambda)\sum_{\tau=1}^\infty \tau\lambda^{\tau}+1
                =\frac{\lambda}{1-\lambda}+1
                =\frac{1}{1-\lambda},
\]
and so $\sigma^2_t\sim (1-\lambda)\,t$ is the expected asymptotic 
behavior for the MSD. If $\sigma_t^2=(1-\lambda)\,t+\epsilon_t$, 
then, according to Equation~\eqref{eq:epsilon} we have
\begin{align*}
\epsilon_{t+1}
&=(1-\lambda)\sum_{0\leq \tau\leq t}\lambda^{\tau}
                                   \epsilon_{t-\tau}
+(1-\lambda)^2\sum_{\tau> t}\lambda^{\tau}(\tau-t)
                 -(1-\lambda)\sum_{\tau>t}\lambda^\tau\\
&=(1-\lambda)\sum_{0\leq \tau\leq t}\lambda^{\tau}
                                    \epsilon_{t-\tau}
+(1-\lambda)^2\lambda^{t}\sum_{\tau\geq 0}\lambda^{\tau}
\tau-\lambda^{t+1}=(1-\lambda)
\sum_{0\leq \tau\leq t}\lambda^{\tau}\epsilon_{t-\tau}.
\end{align*}
Since $\epsilon_0=0$, it can be easily checked, by induction for 
instance, that the previous recursion has solution $\epsilon_t=0$ 
for all $t\in{\mathbb N}_0$ and in this case 
\[
\sigma^2_t=(1-\lambda)\,t.
\] 
We have therefore an exact diffusive behavior for all times.
 
\bigskip
\paragraph{\emph{Power-law trapping time}.} 
Another interesting example is given by power-law distributed trapping times. Let us assume that $p(\tau)=(\tau+1)^{-q}/\zeta(q)$ for each $\tau$. For $2 < q \leq 3$ we have 
\[\mathbb{E}(T)=\zeta(q-1)/\zeta(q)-1<\infty\] but $\mathbb{E}(T^2)=\infty$. In this case we have
\begin{align*}
R_t&:=\frac{\sum_{\tau\leq t+1} p(\tau)\,\tau(\tau+1)/2
           }{\mathbb{E}(T)+1}+(t+1)
      \frac{\sum_{\tau  > t+1} (\tau-t/2)\,p(\tau)
           }{\mathbb{E}(T)+1}\\
    &=\frac{\sum_{\tau\leq t+1}
      \left((\tau+1)^{-q+2}-(\tau+1)^{-q+1}\right)
           }{2\,\zeta(q-1)} 
    + \frac{(t+1)\sum_{\tau >t+1}
      \left((\tau+1)^{-q+1}-(t/2+1)(\tau+1)^{-q}\right)
           }{\zeta(q-1)}.
\end{align*}
From here it is easy to deduce that $R_t\geq c (t+2)^{3-q} -c_0$ for each 
$t\geq 0$, for some constants $c, c_0$ that depend on $q$. Taking this into account, according to Theorem~\ref{teo:NormalMSD} we have
\[
\sigma_t^2\geq\frac{\zeta(q)\,t
                  }{\zeta(q-1)}+c\,(t+2)^{3-q}-c_0
\]
for each $t\geq 0$. Hence, in this case the diffusive behavior suffers an 
increasingly diverging sub-linear deviation of the order of $t^{3-q}$. In this 
case, although $\sigma_t^2$ is asymptotically dominated by a linear behavior, the 
sub-linear deviation on $\sigma_t^2$ is enough to produce an effective sub-
diffusion at all finite times. Indeed, using the recurrence established in 
Proposition~\ref{prop:sigmarecurrence} we computed $\sigma_t^2$ for 
$1\leq t\leq 2^{17}$, and fitted it to a power-law 
$\sigma_t^2={\mathcal O}(t^\beta)$, with 
exponent $\beta=\beta_N(q)$ depending on the time interval chosen for the fit. 
As expected, $\beta(q)\to 1$ as the length of the fitting time interval 
increases. We chose the time interval $10\leq t\leq N$, with $N=2^{13}, 2^{15}$ 
and $2^{17}$, to compute the power-law approximation. The exponents $\beta_N(q)$ 
we obtained are plotted in Figure~\ref{fig:exponent>1}. For each $N$, the 
behavior of $\beta_N(q)$ can be very well fitted by a sigmoidal function 
$\beta(q)=2/(1+ e^{r|q-3|^\eta} )$.
\begin{center}
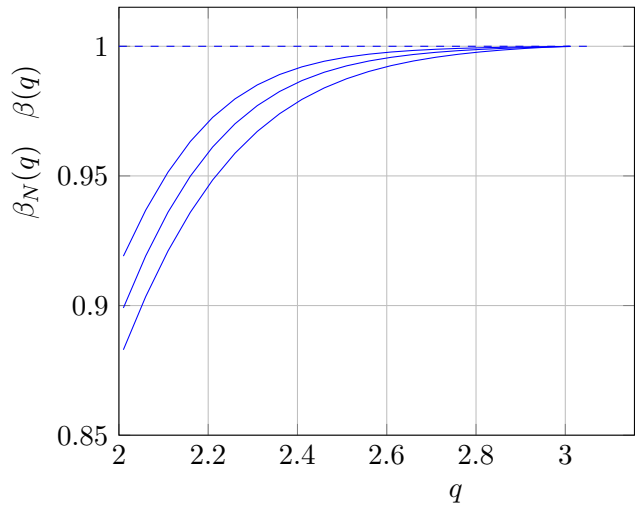
\begin{figure}[h]
\begin{tikzpicture}
\definecolor{mybarcolor}{RGB}{210 207 155}
\pgfmathsetmacro{\xmin}{2}
\pgfmathsetmacro{\ymin}{0.85}
\begin{axis}[ymin=0,%
  xmin=\xmin,%
  ymin=\ymin,
  ylabel=\textbf{$\hskip 60pt \beta_{N}(q)$\hskip 10pt $\beta(q)$},%
  xlabel=\textbf{\hskip 60pt  $q$},%
  grid=major,%
]

\addplot[ color = blue]  table {%
   2.010000000000000   0.919066413639457
   2.060000000000000   0.936791849435091
   2.110000000000000   0.951469421361596
   2.160000000000000   0.963290069284060
   2.210000000000000   0.972575043304649
   2.260000000000000   0.979712690502832
   2.310000000000000   0.985102539369774
   2.360000000000000   0.989115592569341
   2.410000000000000   0.992072349098775
   2.460000000000000   0.994235306508175
   2.510000000000000   0.995811028876723
   2.560000000000000   0.996957249210187
   2.610000000000000   0.997791741933496
   2.660000000000000   0.998401015742406
   2.710000000000000   0.998847881125028
   2.760000000000000   0.999177590488029
   2.810000000000000   0.999422598780808
   2.860000000000000   0.999606134378502
   2.910000000000000   0.999744826660915
   2.960000000000000   0.999850600317113
   3.010000000000000   0.999932036718338
};
\addplot[ color = blue] table{%
   2.010000000000000   0.899087598039155
   2.060000000000000   0.919150508446022
   2.110000000000000   0.936020759513816
   2.160000000000000   0.949917483679236
   2.210000000000000   0.961155241420844
   2.260000000000000   0.970098192597894
   2.310000000000000   0.977120012811324
   2.360000000000000   0.982574258804246
   2.410000000000000   0.986776207262039
   2.460000000000000   0.989994639253832
   2.510000000000000   0.992450890590414
   2.560000000000000   0.994322405493233
   2.610000000000000   0.995748537962847
   2.660000000000000   0.996837030720099
   2.710000000000000   0.997670226444288
   2.760000000000000   0.998310537253536
   2.810000000000000   0.998805007705754
   2.860000000000000   0.999188981868050
   2.910000000000000   0.999488976336654
   2.960000000000000   0.999724886829804
   3.010000000000000   0.999911662996148
};
\addplot[ color = blue] table{%
   2.010000000000000   0.882950070940828
   2.060000000000000   0.903449532142468
   2.110000000000000   0.921062887762581
   2.160000000000000   0.935977869987958
   2.210000000000000   0.948443343357862
   2.260000000000000   0.958742260984164
   2.310000000000000   0.967167835353517
   2.360000000000000   0.974004903350795
   2.410000000000000   0.979517106146872
   2.460000000000000   0.983939467835437
   2.510000000000000   0.987475360812692
   2.560000000000000   0.990296642455081
   2.610000000000000   0.992545817866853
   2.660000000000000   0.994339295836241
   2.710000000000000   0.995771058899079
   2.760000000000000   0.996916302860459
   2.810000000000000   0.997834788345549
   2.860000000000000   0.998573780566671
   2.910000000000000   0.999170542447532
   2.960000000000000   0.999654396883074
   3.010000000000000   1.000048402581453
};

\addplot[domain=\xmin:3.05,samples=100, color=blue, dashed] {1};
\end{axis}
\end{tikzpicture}
\caption{The curves $q\mapsto \beta_N(q)$ correspond to the exponent 
of the approximated power-law behavior of the MSD as a function of the 
trapping time distribution's exponent. We show these curves for total 
observation times $N=2^{13}, 2^{15}$ and $2^{17}$. The curves approaches, 
as $N\to\infty$, the asymptotic exponent $q\mapsto\beta(q)=1$ (dashed line).}
\label{fig:exponent>1}
\end{figure}
\end{center}

\subsection{Central Limit Theorem}\label{subsec:CLT}
Let us now focus on the distribution $\mathbb{P}(X_t=z)$. According 
to what we deduced in Subsection~\ref{subsec:subordinated},
\[
\mathbb{P}(X_t=z)
=\sum_{n=0}^{t-1}\mathbb{P}(N_t=n) \mathbb{P}(S_n=z)
                                             \nonumber\\
= \sum_{n=0}^{t-1}2^{-n}
    \mathbb{P}(N_t=n)\left(\begin{array}{c}n          \\
                            \frac{n+z}{2}\end{array}
                      \right).
\]
On the other hand, $N_t$ can also be related to a sum of i.i.d. 
trapping times. Let us recall that, $N_t:=\#\{0\leq s \leq t:\ T_s=0\}$. 
It can be easily verified that, for any finite collection 
$t_1<t_2<\cdots t_n$ and corresponding trapping times 
$\tau_1,\tau_1\ldots,\tau_n\geq 0$, we have
\[
\mathbb{P}(T_{t_1+1}=\tau_1,\ldots,T_{t_n+1}=
                    \tau_n|\, T_{t_1}=\cdots =T_{t_n}=0)
                    =
   \prod_{k=1}^n\mathbb{P}(T_{t_k+1}=\tau_k|\, T_{t_k}=0)
                    =\prod_{k=1}^np(\tau_k).
\]
From this it follows that 
\begin{equation}\label{eq:NumberOfHits} 
N_t=\max \left\{ n\geq 0:\, \sum_{k=1}^n  (T_k+1)
\leq t,\right\}
\end{equation}
where the random variables $T_k$ are independent copies of the trapping 
time $T$, which satisfies $\mathbb{P}(T=\tau)=p(\tau)$. 
Equation~\eqref{eq:subordinatedrw} suggests that the shape of the 
probability distribution $\mathbb{P}(X_t=z)$ strongly depends on the 
behavior of the distribution of the sum $\sum_{k=0}^n(T_k+1)$. We  
now consider the case where $\mathbb{E}(T^\alpha)<\infty$ for some 
$\alpha > 1$. In this case the distribution of $N_t$ concentrates around 
its mean, and this ensures the following.

\begin{theorem}[Central Limit Theorem for $\alpha > 1$]
\label{teo:CentralLimitTheorem}
Let us assume that\, $\mathbb{E}(T^\alpha)<\infty$ for some 
$\alpha\in (1,2]$ and define $C_T:=\mathbb{E}(|T-\mathbb{E}(T)|^\alpha)$ 
and $\mu=\mathbb{E}(T)+1$. For each $C>8\,C_T/\mu+\sqrt{2\,\mu/\pi}$, 
there exists $t^*=t^*(C)\in \mathbb{N}$ such that for each bounded 
interval $I\subset \mathbb{R}$ and each $t\geq t^*$, we have
\[
\left|\mathbb{P}
\left(X_t\in \sqrt{\frac{t}{\mu}}\, I\right)
    -\frac{1}{\sqrt{2\pi}}\int_{I}e^{-x^2/2}dx\right|\leq
C\,t^{\frac{1-\alpha}{1+\alpha}}
\]
\end{theorem}

\medskip \noindent 
The proof of this result relies on two classical inequalities. On the one 
hand the Berry-Esseen inequality (as it appears in~\cite{2007Hip&Mattner}), 
which gives the rate of convergence in the de Moivre-Laplace theorem. 
According to it,
\begin{equation}\label{eq:berryessen}
\sup_{x\in\mathbb{R}}\left|2^{-n}\sum_{z\leq x}
  \left(\begin{array}{c}n\\ \frac{n+z}{2}\end{array}\right)
   -\frac{1}{\sqrt{2\pi}}\int_{-\infty}^{x/\sqrt{n}} 
   e^{-z^2/2}\, dz\right|<\frac{1}{\sqrt{2\pi\,n}}
\end{equation}
for all $n\in \mathbb{N}$. On the other hand, we use a generalization 
of Chebyshev's inequality, due to von Bahr and Esseen. It states that 
if $X_k$, $1\leq k\leq n$ are i.i.d. random variables such that 
$\mathbb{E}(X_k)=0$ and $\mathbb{E}(|X_k|^\alpha):=C_\alpha<\infty$, 
for some $\alpha\in [1,2]$, then 
\begin{equation}\label{eq:vanBahrEsseen}
\mathbb{P}\left(\left|\sum_{k=0}^nX_k\right| > \epsilon\right)
\leq \frac{2\, n\, C_\alpha}{\epsilon^\alpha}.
\end{equation}
The inequality was established in~\cite{1965vonBahr&Esseen}. 

\noindent
\begin{proof}
It follows from Eq.~\eqref{eq:NumberOfHits} that 
\begin{align}
\mathbb{P}\left(N_t\leq n\right)
   &=\mathbb{P}\left(\sum_{k=1}^{n+1}(T_k+1)\geq t+1
              \right),         \label{eqs:psum1}\\
\mathbb{P}\left(N_t\geq n\right)
   &=1-\mathbb{P}\left(N_t\leq n-1\right) \nonumber\\
   &=1-\mathbb{P}\left(\sum_{k=1}^{n}(T_k+1)\geq t+1\right)
 =\mathbb{P}\left(\sum_{k=1}^{n}(T_k+1)\leq t\right).
\label{eqs:psum2}
\end{align}
for each $n\in\mathbb{N}$.

\medskip
\noindent
Let $n_1:=(t+1-d_1)/\mu-1$ and $n_2:=(t+d_2)/\mu-1$ be numbers, with $d_1,d_2$ 
yet to be specified. Let $d=\min(d_1,d_2)$, then, taking~\eqref{eqs:psum1} 
and~\eqref{eqs:psum2} into account, we obtain
\begin{align*}
\mathbb{P}\left(N_t\notin (n_1,n_2)\right) 
&\leq \mathbb{P}
   \left(\left|\sum_{k=1}^{n_1+1}(T_k+1)-(n_1+1)\mu
                               \right|\geq d_1\right) +  
       \mathbb{P}
   \left(\left|\sum_{k=1}^{n_2+1}(T_k+1)-(n_2+1)\mu
                             \right|\geq d_2\right)\\
&\leq \frac{2(n_1+n_2+2)\, C_T}{\mu\,d^\alpha} 
     =\frac{2(2\,t+1+d_2-d_1)\, C_T}{\mu\,d^\alpha}
\end{align*} 
for each $t\in \mathbb{N}$. From now on we will assume that 
$t > (1+\mu)^{1/\beta}\,2^{-1/\beta}$, which ensures that $n_1 < n_2$. 
Now, by taking $\beta \in (1/\alpha,1)$ and $d_1, d_2$ such that 
$|t^\beta-d_1|,|t^\beta-d_2|\leq \mu$, we obtain
\[
\mathbb{P}\left(N_t\notin (n_1,n_2)\right)
        \leq \frac{2(2+(1+2\mu)t^{-1})\, C_T
        }{\mu\,(1-\mu\,t^{-\beta})^{\alpha}}\,t^{1-\alpha\,\beta}.
\]
From now on, $g_1(t):=(1+1/2(1+2\mu)t^{-1})/(1-\mu\,t^{-\beta})^{\alpha}$. 
Using the previous inequality and Berry-Esseen's inequality, it follows that
\begin{align}\label{eq:concentration}
\left|
\mathbb{P}\left(X_t\in \sqrt{\frac{t}{\mu}}\, I\right) 
  - \frac{1}{\sqrt{2\pi}} \sum_{n=n_1}^{n_2} 
  \mathbb{P}\left(N_t=n\right)
  \int_{\sqrt{\frac{t}{\mu\,n}}\, I} e^{-z^2/2}\,dz \right|                 
&\leq g_1(t)\,\frac{4\,C_T}{\mu}\,t^{1-\alpha\,\beta}
         +\frac{2}{\sqrt{2\pi\,n_1}}   \nonumber\\
&=g_1(t)\,\frac{4\,C_T}{\mu}\,t^{1-\alpha\,\beta}
        +g_2(t)\,\sqrt{\frac{2\,\mu}{\pi\,t}}
\end{align}
where $g_2(t):=1/\sqrt{1-t^{\beta-1}-(2\mu-1)t^{-1}}$. 

\medskip\noindent 
For $n\in (n_1,n_2)$ we have
\begin{align*}
\left|\int_{\sqrt{\frac{t}{\mu\,n}}\, I} e^{-z^2/2}\,dz
              -\int_{I} e^{-z^2/2}\,dz \right|
&\leq 
\int_{\sqrt{\frac{t}{\mu\,n}}\,I\,\Delta\, I}e^{-z^2/2}\,dz 
\leq \sqrt{\frac{t}{n_1\mu}}
          \left(1-\sqrt{\frac{n_1}{n_2}}\right)     \\
&\leq \frac{1}{\sqrt{1-t^{\beta-1}-(2\mu-1)t^{-1}}}
 \left(1-\sqrt{\frac{1-t^{\beta-1}-(2\mu-1)t^{-1}
             }{1+t^{\beta-1}+2\mu\,t^{-1} }}\right) \\
&\leq 2(t^{\beta-1}+(4\mu-1)\,t^{-1}),                              
\end{align*}
provided 
$t\geq t_1:= \min\left\{t\in{\mathbb N}:\ 
(t^{\beta-1}+(4\mu-1)\,t^{-1})(t^{\beta-1}+(2\mu-1)t^{-1}) \leq 1\right\}$. 
Finally, using this in~\eqref{eq:concentration}, and taking into account that 
$\mathbb{P}\left(N_t\notin (n_1, n_2)\right)\leq 
                    g_1(t)\,(4\,C_T/\mu)\,t^{1-\alpha\,\beta}$, 
we finally obtain
\begin{align*}
\left|
\mathbb{P}\left(X_t\in \sqrt{\frac{t}{\mu}}\, I\right)
        -\int_{I} e^{-z^2/2}\,dz \right|
&\leq g_1(t)\,\frac{8\,C_T}{\mu}\,t^{1-\alpha\,\beta}+
           g_2(t)\,\sqrt{\frac{2\mu}{\pi\,t}}+g_3(t)\,
                      \sqrt{\frac{2}{\pi}}\,t^{\beta-1},
\end{align*}
with $g_3(t):=(1+(4\mu-1)\,t^{-\beta})$. Now, optimizing 
$\beta \in (1/\alpha,1)$ we obtain
\[
\left|\mathbb{P}\left(X_t\in \sqrt{\frac{t}{\mu}}\, I\right)
           -\int_{I} e^{-z^2/2}\,dz \right|
\leq \left(g_1(t)\,\frac{8\,C_T}{\mu}
+g_2(t)\, 
\sqrt{\frac{2\mu}{\pi}}\,t^{-\frac{3-\alpha}{2(\alpha+1)}}
+g_3(t)\,\sqrt{\frac{2}{\pi}}
\right)\,t^{\frac{1-\alpha}{1+\alpha}}.
\]
The theorem follows by taking into account that $\max(g_1(t),g_2(t),g_3(t))\to\infty$ when $t\to\infty$.
\end{proof}

\bigskip
\begin{remark}\label{rem:chernoff}
If the distribution of trapping times has exponential moments, i.e., if $\mathbb{E}(\exp(q\,T))<\infty$ for some $q\geq 0$, then the upper 
bound on the rate of convergence can be improved by replacing van 
Bahr-Esseen's inequality by a tighter inequality, using a Chernoff's 
bound. In that case we can take $|\mu\,n_1-t|$ and 
$|\mu\,n_2-t|$ of order $\sqrt{t\,\log(t)}$, which gives a rate of 
convergence of the order of $1/\sqrt{t}$. This is the case of the 
exponential trapping time distribution $p(\tau)=(1-\lambda)\,\lambda^\tau$.
\end{remark}

\bigskip\noindent 
The case $\mathbb{E}(T)<\infty$ and $\mathbb{E}(T^\alpha)=\infty$ for 
each $\alpha>1$ has to be treated separately. In this case we are unable 
to give a general explicit bound for the rate of convergence towards 
normality, since we can no longer use an explicit concentration 
inequality. In this case our proof relies on the attractiveness of 
the Cauchy distribution and applies only to trapping time distributions 
of the kind $p(\tau)=\tau^{-2}\,L(\tau)$ with $L(\tau)$ a slowly 
varying function. 

\begin{theorem}[CTL for $\alpha=1$]
\label{teo:CentralLimitTheoremOne}
Let $\mathbb{E}(T)<\infty$ and define $\mu=\mathbb{E}(T)+1$
as above. If in addition $\lim_{\tau\to\infty}
\mathbb{P}(T \geq t\, \tau)/\mathbb{P}(T \geq \tau)=t^{-1}$
for each $t\in \mathbb{R}^+$, then  
\[
\lim_{t\to\infty}
\left|\mathbb{P}
\left(X_t\in \sqrt{ \frac{t}{\mu}}\, I\right)
    -\frac{1}{\sqrt{2\pi}}\int_{I}e^{-x^2/2}dx\right|=0.
\]
\end{theorem}

\medskip\begin{remark}\label{rem:alpha=1}
Note that 
$\lim_{\tau\to\infty} 
\mathbb{P}(T \geq t\, \tau)/\mathbb{P}(T \geq \tau)=t^{-1}$ for all 
$t\in \mathbb{R}^+$ implies that $\mathbb{E}(T^\alpha)=\infty$ for each 
$\alpha>1$. For this it is enough to notice that if 
$\sum_{\tau\in\mathbb{N}}\tau^{1+\epsilon}p(\tau) < \infty$ for some 
$\epsilon > 0$, then 
$\sum_{\tau\in \mathbb{N}}\tau^{\epsilon} \, 
\mathbb{P}(T\geq \tau)< \infty$, and necessarily 
\[
\liminf_{\tau\to\infty}
\frac{(t\,\tau)^{\epsilon}\,\mathbb{P}(T\geq t\,\tau)
    }{\tau^{\epsilon}\,\mathbb{P}(T\geq \tau)}= 
\liminf_{\tau\to\infty}
\frac{t^{\epsilon}\,\mathbb{P}(T\geq t\,\tau)
    }{\mathbb{P}(T\geq \tau)}=c\in (0,1),
\]
for each $t\in\mathbb{R}^+$, and therefore 
$\liminf_{\tau\to\infty}
\mathbb{P}(T \geq t\, \tau)/\mathbb{P}(T \geq \tau)=
c\,t^{-\epsilon}$ for all $t\in \mathbb{R}^+$. 
\end{remark}

\begin{proof} As we mentioned above, the proof of this result relies 
on the attractiveness of the Cauchy distribution. Indeed, the condition 
$\lim_{n\to\infty}\mathbb{P}(T \geq t\, n)/\mathbb{P}(T \geq n)=t^{-1}$ 
ensures that $T$ belongs to the domain of attraction of a stable 
distribution with index $\alpha=1$. Furthermore, since $T>0$, then 
the stable limiting distribution, which necessarily has characteristic 
function $t\mapsto \exp(-|t|)$, is nothing but the Cauchy distribution, 
$x\mapsto \pi/(1+x^2)$ (see~\cite{2000Geluk&Haann} for instance). 
Therefore, according to Gnedenko's Theorem~\cite{1968Gnedenko&Kolmogorov}, 
there are sequences $\{a_n\}_{n\in\mathbb{N}}$ in $\mathbb{N}$ and 
$\{b_n\}_{n\in\mathbb{N}}$ in $\mathbb{R}$ such that 
$\sum_{k=1}^n\, T_k/a_n-b_n$ converges in law to 
$x \mapsto \pi^{-1}\int_{t\leq x}dt/(1+t^2)$. Furthermore, we can take 
$a_n:=\min\{\tau\in\mathbb{N}:\, 
             \mathbb{P}(T > \tau)\leq 1/n\}$ 
and 
$b_n=n\,(\sum_{\tau \leq a_n}\tau\,p(\tau))/a_n$ (see for instance 
Theorem 3.7.4 in~\cite{2010Durrett}). It follows that there exists 
a positive sequence $\{\epsilon_n\}_{n\in\mathbb{N}}$ converging to 
zero, such that for each $x\in \mathbb{R}$ and $n\in\mathbb{N}$ we have
\[
\mathbb{P}\left(\frac{\sum_{k=1}^n\, T_k}{a_n}
   -b_n \leq t \right)
   \leq \frac{1}{\pi}\int_{-\infty}^t\frac{dx}{1+x^2}
   +\epsilon_n.
\]
Let $t\mapsto d(t)$ be a positive diverging function, yet to be specified, 
and define
\[
  n_1:=\max\{n\leq t/\mu:\,t > n+a_{n}\,(b_{n}+d(t))\},
       \hskip 10pt
  n_2:=\min\{n\geq t/\mu:\,t < n+a_{n+1}\,(b_{n+1}-d(t))\}.
\]
With this we have
\begin{align*}
\mathbb{P}\left(N_t\leq n_1\right)=
\mathbb{P}\left(\sum_{k=1}^{n_1}\, T_k \geq t-n_1 \right)
&\leq\frac{1}{\pi}
   \int_{d(t)}^\infty\frac{dx}{1+x^2}+\epsilon_{n_1},\\
&=\frac{1}{2}-\frac{\arctan(d(t))}{\pi}+\epsilon_{n_1}\\
   \mathbb{P}\left(N_t\geq n_2\right)=
   \mathbb{P}\left(\sum_{k=1}^{n_2+1}\, T_k\leq t-n_2\right)
&\leq\frac{1}{\pi}
   \int_{-\infty}^{d(t)}\frac{dx}{1+x^2}+\epsilon_{n_2+1}\\
&=\frac{1}{\pi}\int_{d(t)}^\infty\frac{dx}{1+x^2}+
   \epsilon_{n_2+1}\\
&=\frac{1}{2}-\frac{\arctan(d(t))}{\pi}+\epsilon_{n_2+1}.
\end{align*}
Therefore 
\begin{equation}\label{eq:arctan}
\mathbb{P}\left(N_t\notin (n_1,n_2)\right)
\leq 2\left(\epsilon_{n_1}+
      \frac{1}{2}-\frac{\arctan(d(t))}{\pi}
       \right)\leq 
  2\left(\epsilon_{n_1}+\frac{2}{d(t)}\right),
\end{equation}
provided $d(t)$ is sufficiently large. Following exactly the same 
computations as in the proof of Theorem~\ref{teo:CentralLimitTheorem}, 
we obtain
\begin{align*}
\left|
\mathbb{P}\left(X_t\in \sqrt{\frac{t}{\mu}}\, I\right)
     -\int_{I} e^{-z^2/2}\,dz \right|
\leq 4\left(\epsilon_{n_1}+\frac{2}{d(t)}\right)+
\sqrt{\frac{t}{n_1\mu}}
    \left(1-\sqrt{\frac{n_1}{n_2}}\right)
                   +\frac{2}{\sqrt{2\pi\,n_1}}
\end{align*}
provided $\mu\, n_1\leq t \leq \mu\, n_2$. 

\medskip\noindent To complete the proof we need to find a diverging 
function $t\mapsto d(t)$ such that $n_2/n_1\rightarrow 1$ as 
$t\to\infty$. For this, consider $g(\tau):=\tau\,\mathbb{P}(\tau)$. 
Since $g(\tau) \leq \sum_{\tau'\geq\tau}\tau\,p(\tau')$ and 
$\mathbb{E}(T)<\infty$, then $\lim_{\tau\to\infty}g(\tau)=0$. 
On the other hand, by the hypothesis on $\mathbb{P}(T\geq \tau)$, 
for each $\epsilon > 0$ the function 
$\tau\mapsto g_\epsilon(\tau):=\tau^{\epsilon}\,g(\tau)$ is regularly 
varying of order $\epsilon$, i.e.
$\lim_{\tau\to\infty} g_\epsilon(t\,\tau)/g_\epsilon(\tau)=t^{\epsilon}$ 
for each $t\in \mathbb{R}^+$. Therefore $g_\epsilon(\tau)$ diverges for 
all $\epsilon >0$, which proves that $g(\tau)$ converges to zero slower 
than any power-law. Now, since by definition 
$g(n)\geq a_n/n\geq g(n+1)-\mathbb{P}(T\geq a_n+1)$, then 
$\{n/a_n\}_{n\in\mathbb{N}}$ is a diverging sequence growing slower 
than any power-law, and indeed $n\mapsto n/a_n$ is a slowly varying 
function, i.e. a regularly varying function of order zero. Taking all 
this into account, it follows that the function
\[
 n\mapsto h(n):=\min\left\{\frac{\tau}{a_{\tau}}:\, 
      \tau\geq n\right\},
\]
diverges monotonously and slower than any power-law. Furthermore, 
it is a slowly varying function. Finally, let  $d(t):=\sqrt{h([t/\mu])}$, 
which is monotonously diverging and slow varying as well. To conclude, 
notice that from the definition of $n_1$, $n_2$, and $h(t)$, we have on 
the one hand,
\[
n_1+1\geq \frac{t}{\mu}-\frac{a_{n_1+1}\,d(t)}{\mu}
   \geq \frac{t}{\mu}-\frac{a_{[t/\mu]+1}\,d(t)}{\mu}
   \geq \frac{t}{\mu}-\frac{(t/\mu+1)\,d(t)}{\mu\,h([t/\mu]+1])}
   \geq \frac{t}{\mu}-\frac{(t/\mu+1)}{\mu\,d(t)}
\]
and on the other hand
\begin{align*}
n_2-1\leq \frac{t}{\mu}+\frac{a_{n_2}\,d(t)}{\mu}+\frac{n_2\, 
  \sum_{\tau > n_2} \tau\,p(\tau)}{\mu}
 &\leq \frac{t}{\mu}+\frac{n_2\,d(t)}{\mu\,h(n_2)}+\frac{n_2\, 
     \sum_{\tau > n_2} \tau\,p(\tau)}{\mu}\\
 &\leq \frac{t}{\mu}+\frac{n_2\,d(t)}{\mu\,h([t/\mu])}
    +\frac{n_2\, \sum_{\tau > t/\mu} \tau\,p(\tau)}{\mu},
\end{align*}
and therefore
$n_2\left(1-(\mu\,d(t))^{-1}-
   \mu^{-1}\sum_{\tau > t/\mu} \tau\,p(\tau)
                               \right)-1\leq t/\mu$.
Hence, 
\[
1\leq \frac{n_2}{n_1}
\leq  \frac{(t/\mu+1)
\left(1-(\mu\,d(t))^{-1}-\mu^{-1}\sum_{\tau > t/\mu}\tau\,p(\tau)
\right)^{-1}
 }{t/\mu\left(1-\left(1+\mu\,t^{-1}\right)
              \left(d(t)\,\mu\right)^{-1} \right) }
 \rightarrow 1, \text{ as } t\to\infty,
\] 
and the proof is done.
\end{proof}

\bigskip
\section{Sub-diffusive regime}\label{sec:subdiffusive}
\noindent 
In this section we study trapping time distributions leading to a 
sub-normal diffusion. First, we will analyze the sub-linear growth 
of the MSD in the case of a trapping time distribution with diverging 
first moment. In the particular case of power-law distributed trapping 
times, we are able to determine the behavior of this sub-linear growth, 
which turns out to be a power-law as well. Then, we study the 
concentration of the number of steps made by the trapped random walker, 
for the larger class of distributions for which 
$\mathbb{E}(T^\alpha)<\infty$, for some $\alpha < 1$.

\subsection{Sub-diffusion via MSD}
\label{subsec:subnormalMSD} 
Let us consider power-law decaying probability distributions for which 
the mean trapping time diverges. In this case we lose the asymptotic 
normal behavior since the MSD now follows a sub-linear power-law growth. 
Let us start by establishing the sub-linearity of the MSD with respect 
to time, which holds in all cases when $\mathbb{E}(T)=\infty$.

\medskip
\begin{proposition}[Sub-diffusion]\label{prop:subdiffusion}
Suppose that the trapped random walk satisfies the initial condition~\eqref{eq:initialcondition}. If the trapping-time has infinite 
mean, then $\lim_{t\to\infty} \sigma_t^2/t=0$.
\end{proposition}

\begin{proof}
Let us assume that $\sigma_t^2=\alpha\,t+\epsilon_t$ for some 
$\alpha >0$ and each $t\geq 0$. Using~\eqref{eq:epsilon} we obtain 
$\epsilon_{t+1}=\sum_{0\leq \tau\leq t}\epsilon_{t-\tau}  
   p(\tau)+\delta_t$, 
with
\[
\delta_t= 
\mathbb{P}(T\leq t)-\alpha \left((t+1)\,\mathbb{P}(T>t)
           +\sum_{0\leq\tau\leq t}(\tau+1)p(\tau)\right).
\]
Following the same reasoning as in the proof of 
Theorem~\ref{teo:NormalMSD} we obtain
$\epsilon_{t+1}=\sum_{0\leq \tau\leq t}Q_{t-\tau}
                             \delta_\tau$, with $Q_{\tau}$ 
recursively defined by 
$Q_{t+1}=\sum_{0\leq\tau\leq t}Q_{t-\tau}\,p(\tau)$, and such that 
$Q_0=1$. Taking this into account we have 
\begin{align*}
\epsilon_{t+1}=\sum_{0\leq \tau\leq t}Q_{t-\tau}\delta_\tau
&\geq - \alpha \left(\sum_{0\leq \tau\leq t}(\tau+1)
         \mathbb{P}(T>\tau) + \sum_{0\leq\tau\leq t}
\sum_{0\leq \tau'\leq \tau}(\tau'+1)p(\tau') \right)\\
&\geq -\alpha\,(t+1) 
      \left(\frac{t+2}{2}\mathbb{P}(T\geq t+1)
+\sum_{0\leq\tau\leq t}(\tau+1)\,p(\tau)\right),
\end{align*}
for all $t\in\mathbb{N}$. 
Since $\mathbb{E}(T)=\infty$, then 
$\lim_{t\to\infty}\delta_t=-\infty$, which implies that
$\delta_t < 0$ eventually for all $t$. Hence, there
exists $t_1\in \mathbb{N}$ such that 
$\epsilon_{t+1}=\sum_{0\leq \tau\leq t}
Q_{t-\tau}\delta_\tau \leq 
\sum_{0\leq \tau < t_1}|\delta_\tau|$ 
for all $t\in\mathbb{N}$. Thus, we have the inequalities
\[
-\alpha\,(t+1)\left(\frac{t+2}{2}\mathbb{P}(T\geq t+1)
     +\sum_{0\leq\tau\leq t}(\tau+1)\,p(\tau)\right)
\leq \epsilon_{t+1}\leq 
       \sum_{0\leq \tau < t_1}|\delta_\tau|
\]
for all $t\in\mathbb{N}$, from which it follows that
\[
-\infty \leq \limsup_{t\to\infty}\frac{\sigma_t}{t}
=\alpha+\limsup_{t\to\infty} \frac{\epsilon_t}{t}
\leq \alpha.
\]
The proposition follows from the fact that $\sigma_t^2>0$ for each 
$t\geq 0$, taking into account that $\alpha>0$ can be taken 
arbitrarily small.
\end{proof}

\bigskip\noindent
\paragraph{\emph{Variation slower than any power-law}}
In the case of trapping time distributions decaying as a power-law, 
the growth of the MSD is dominated by a power-law sub-linear growth. 
The exponent of the power-law governing the MSD growth directly depends 
on the exponent of the power-law decay of the trapping time distribution. 
To be more precise we need the following definition. We will say that 
a strictly positive function $\tau\mapsto g(\tau)$ \emph{varies slower than 
any power-law} if $\lim_{t\to\infty}g(t)\,t^{-\epsilon}=0$ and 
$\lim_{t\to\infty}g(t)\,t^{\epsilon}=\infty$ for any $\epsilon>0$. 
In particular, any regularly varying function of order zero varies 
slower than any power-law.

\medskip\noindent
In what follows we will need the next statement with 
some properties satisfied by functions varying slower than any power-law, 
the proof of which is 
remitted to the Appendix~\ref{app:slower}.

\medskip\begin{claim}\label{cla:slowerthan}
Let $t\mapsto g(t)$ and $t\mapsto h(t)$ be two functions varying slower 
than any power-law. Then the following are functions varying slower 
than any power-law.
\begin{center}
\begin{tabular}{lll}
a) $t\mapsto \lambda\,g(t)$ with $\lambda >0$, 
   & b) $t\mapsto g(t)+h(t)$, 
   & c) $t\mapsto g(t)\,h(t)$, \\  
d) $t\mapsto 1/g(t)$,       
   & e) $t\mapsto \min_{\mu\,t\leq \tau\leq \lambda\, t}g(\tau)$,  
   & f) $t\mapsto \max_{\mu\,t\leq \tau\leq \lambda\, t}g(\tau)$
                            with $0\leq \mu < \lambda$.
\end{tabular}
\end{center}
Furthermore, if $g(t)\leq f(t)\leq h(t)$ for each $t\in\mathbb{N}$, then \\
g) $t\mapsto f(t)$ varies slower than any power-law, \\
and, if $\tau\mapsto P(\tau)\geq 0$ is such that
$0 < \sum_{s \geq 1}P(s)\,(s+1)^{\epsilon_0} < \infty$ for some 
                                    $\epsilon_0 >0$, then \\
h) $t\mapsto 
\sum_{s\geq 1} P(s)\,\max_{s\,t < \tau\leq (s+1)\,t}g(\tau)$, 
varies slower than any power-law as well.
\end{claim}

\medskip \noindent We are now able to prove the following.

\medskip
\begin{theorem}\label{teo:MSD-Subdiffusion}
Let suppose that $p(\tau)=(\tau+1)^{-q}\,g(\tau)$, with $1 < q \leq 2$ 
and $\tau\mapsto g(\tau)$ varying slower than any power-law. Then 
there exists $\tau\mapsto h(\tau)$, varying slower than any power-law, 
such that $\sigma_t^2=h(t)\,t^{q-1}$ for all $t\in\mathbb{N}$.
\end{theorem}

\medskip
\begin{proof}
Let $\beta=\beta(q):=
\sup\{b\geq 0:\ \limsup_{n\to\infty}\sigma_t^2/t^b=\infty\}$. 
Proposition~\ref{prop:unboundedness} ensures that 
$\lim_{t\to\infty}\sigma_t^2=\infty$, which implies $\beta(q)\geq 0$. 
It is easy to verify that $\lim_{t\to\infty}\sigma_t^2/t^b=\infty$ 
for each $0\leq b < \beta$ and $\lim_{t\to\infty}\sigma_t^2/t^b=0$ 
for each $b>\beta$. 

\medskip\noindent For $\beta$ defined as above, the function 
$t\mapsto h(t):=\sigma_t^2/t^{\beta}$ varies slower than any power-law. 
This follows from Claim~\ref{cla:slowerthan}-$h$, and the fact that 
both $h_1(t):=\min_{0\leq \tau\leq t}\sigma_\tau^2/\tau^\beta$ and $h_2(t):=\max_{0\leq \tau\leq t}\sigma_\tau^2/\tau^\beta$ vary slower 
than any power-law. Indeed, for each $\epsilon >0$ we have
\begin{align*}
\lim_{t\to\infty}
h_1(t)\,t^{-\epsilon}
&=
\lim_{t\to\infty}
  \left(
      \min_{1\leq \tau\leq t}
       \left\{\frac{\sigma_\tau^2}{\tau^\beta}
       \right\}\,
      \min_{1\leq \tau\leq t}
        \left\{\tau^{-\epsilon}
        \right\}
  \right)
\leq 
\lim_{t\to\infty}
  \left(
    \min_{1\leq \tau\leq t}
         \frac{\sigma_\tau^2}{\tau^{\beta+\epsilon}}
  \right)
\leq
\lim_{t\to\infty}
   \frac{\sigma_t^2}{t^{\beta+\epsilon}}=0, \\
\lim_{t\to\infty}h_2(t)\,t^\epsilon
&=
\lim_{t\to\infty}
  \left(
    \max_{1\leq \tau\leq t}
    \left\{
       \frac{\sigma_\tau^2}{\tau^\beta}
    \right\}\,
    \max_{1\leq \tau\leq t}
    \left\{\tau^\epsilon\right\}
  \right)
\geq 
\lim_{t\to\infty}
   \left(
     \max_{1\leq \tau\leq t}
      \left\{
        \frac{\sigma_\tau^2}{\tau^{\beta-\epsilon}}
      \right\}
   \right)
\geq 
\lim_{t\to\infty}\frac{\sigma_t^2}{t^{\beta-\epsilon}}=\infty.
\end{align*}
On the other hand, taking into account that 
$\inf_{t\in\mathbb{N}}\sigma_t^2/t^{b}>0$ for each $b<\beta$, and 
$\sup_{t\in\mathbb{N}}\sigma_t^2/t^{b}<\infty$ for each $b>\beta$, 
we have
\begin{align*}
\lim_{t\to\infty}h_1(t)\,t^{\epsilon}
&=
\lim_{t\to\infty}
   \left(
       \min_{1\leq \tau\leq t}
       \left\{
         \frac{\sigma_\tau^2}{\tau^\beta}
       \right\}\,
       t^{\epsilon}
   \right)
\geq 
\lim_{t\to\infty}
    \left(
      \min_{1\leq \tau\leq t}
       \left\{
           \frac{\sigma_\tau^2}{\tau^{\beta-\epsilon/2}}
       \right\}\,
       t^{\epsilon/2}
    \right)\\
&\geq 
  \left(
   \inf_{t\in\mathbb{N}}
      \frac{\sigma_t^2}{t^{\beta-\epsilon/2}}
  \right)
   \lim_{t\to\infty}t^{\epsilon/2}=\infty,\\
\lim_{t\to\infty}h_2(t)\,t^{-\epsilon}
&=
\lim_{t\to\infty}
  \left(
    \max_{1\leq \tau\leq t}
    \left\{\frac{\sigma_\tau^2}{\tau^\beta}\right\}\,
     t^{-\epsilon}
  \right)
\leq 
\lim_{t\to\infty}
  \left(
     \max_{1\leq \tau\leq t}
      \left\{\frac{\sigma_\tau^2}{\tau^{\beta+\epsilon/2}}
      \right\}\,
      t^{-\epsilon/2}
   \right)\\
&\leq 
  \left(
   \sup_{t\in\mathbb{N}}\frac{\sigma_t^2}{t^{\beta+\epsilon/2}}
  \right)
  \lim_{t\to\infty}t^{-\epsilon/2}=0.
\end{align*}

\medskip\noindent To conclude the proof we verify that $\beta(q)=q-1$. 
For this, according to Proposition~\ref{prop:unboundedness}, we note that
$\sigma_s^2-\sigma_t^2\geq 0$ whenever $s\geq t$. Then, using this and 
Proposition~\ref{prop:sigmarecurrence} we obtain
\begin{equation}\label{eq:deltasigma}
\mathbb{P}(T>t)\,\sigma_t^2 = \mathbb{P}(T\leq t)
       -\sum_{\tau\leq t}
       \left(\sigma_t^2-\sigma_{t-\tau}^2\right)\,p(\tau)-
       \left(\sigma_{t+1}^2-\sigma_t^2\right).
\end{equation}
for each $t\geq 1$. This implies that $\sigma_t^2\leq 1/\mathbb{P}(T>t)$ 
and therefore, for each $b > q-1$ and $t\in\mathbb{N}$ we have
\begin{align*}
\frac{\sigma_t^2}{t^b}
& \leq
   \frac{t^{-b+q-1}}{\sum_{\tau>t}g(\tau)\,t^{q-1}\,(\tau+1)^{-q}}
\leq
  \frac{t^{-b+q-1}}{\left(\sum_{t < \tau\leq 2\,t}
      g(\tau)\,(\tau/t+1)^{-q}\right)\,t^{-1}}\\
&\leq  \frac{t^{-b+q-1}}{\min_{t<\tau\leq 2\,t}g(\tau)\,
        \left(\sum_{t < \tau\leq 2\,t}(\tau/t+1)^{-q}\right)\,t^{-1}}
\leq \frac{t^{-b+q-1}}{3^{-q}\,\min_{t<\tau\leq 2\,t}g(\tau)}
\end{align*}
Since by Claim~\ref{cla:slowerthan}-$e$, 
$t\mapsto \min_{t<\tau\leq 2\,t}g(\tau)$ varies slower than any power-law, 
then $\lim \sigma_t^2/t^{b}=0$ for each $b > q-1$, and so $\beta(q)\leq q-1$. 

\medskip\noindent
Let us assume that $q < 2$. In this case, using $\sigma_t^2=h(t)\,t^\beta$ 
in Equation~\eqref{eq:deltasigma}, and taking into account that 
$\beta\leq q-1$ and $p(\tau)=g(\tau)\,(\tau+1)^{-q}$, we obtain
\begin{align}\label{eq:sigmaupperq2}
\mathbb{P}(T>t)\,\sigma_t^2 
&\geq \mathbb{P}(T\leq t)-\max_{\tau\leq t}h(\tau)\left(
         \sum_{\tau\leq t}
           \left(1-\left(1-\tau/t\right)^\beta
           \right)\, p(\tau)
        +\left(1-\left(1-1/t\right)^\beta\right)\right)\,t^{\beta} \nonumber\\  
&\geq \mathbb{P}(T\leq t)-\max_{\tau\leq t}h(\tau)\left(       
         \sum_{\tau\leq t}\tau\, p(\tau)+1\right)\,t^{\beta-1}  \\
&\geq \mathbb{P}(T\leq t)-\max_{\tau\leq t}h(\tau)\,
         \left(\max_{\tau\leq t}g(\tau)\left(1+\frac{(t+1)^{2-q}}{2-q}\right)+1
         \right)\,t^{\beta-1} \nonumber\\
&\geq \mathbb{P}(T\leq t)-\max_{\tau\leq t}h(\tau)\,
         \left(\max_{\tau\leq t}g(\tau)\left(1+\frac{1}{2-q}\right)+1
         \right)\,(t+1)^{\beta+1-q}.  \nonumber
\end{align}
If $\beta=q-1$, then we have nothing to prove. Otherwise, if 
$\beta < q-1$, since by Claim~\ref{cla:slowerthan}-$f$, 
$t\mapsto \max_{\tau\leq t}h(\tau)$ and 
$t\mapsto \max_{\tau\leq t}g(\tau)$ both vary slower than any power-law, 
then there exists $t_0$ such that for all $t\geq t_0$ we have 
$2\,\sigma_t^2\geq 1/\mathbb{P}(T>t)$. Hence, for each $b < q-1$ and 
$t\geq t_0$ we have
\begin{align*}
\frac{\sigma_t^2}{t^b}
& \geq
   \frac{t^{-b+q-1}}{\sum_{s\geq 1}s^{-q}
     \left(\sum_{s\,t < \tau \leq (s+1)\,t}g(\tau)\right)\,t^{-1}}
\geq
  \frac{t^{-b+q-1}}{\sum_{s\geq 1}s^{-q}
     \,\max_{s\,t < \tau \leq (s+1)\,t}g(\tau)}.
\end{align*}
Claim~\ref{cla:slowerthan}-$h$ ensures that $t\mapsto\sum_{s\geq 1}s^{-q}
     \,\max_{s\,t < \tau \leq (s+1)\,t}g(\tau)$ varies slower than 
any power-law, therefore $\lim \sigma_t^2/t^{b}=\infty$ for each 
$b < q-1$, hence $\beta(q)\geq q-1$. For the remaining case, $q=2$, Inequality~\eqref{eq:sigmaupperq2} implies
\[
\mathbb{P}(T>t)\,\sigma_t^2\geq \mathbb{P}(T\leq t) - 
     \max_{\tau\leq t}h(\tau)\,\left(\max_{\tau\leq t}g(\tau)
     \left(1+\log(t+1)\right)+1\right)\,t^{\beta-1},
\]
and since $t\mapsto \max_{\tau\leq t}h(\tau)$, 
$t\mapsto \max_{\tau\leq t}g(\tau)$ and $t\mapsto \log(t+1)$ vary 
slower than any power-law, then $2\,\sigma_t^2\geq 1/\mathbb{P}(T>t)$ 
for all $t$ sufficiently large, and from there we can proceed as in 
the previous case, concluding that $\beta(2)\geq 1$. This completes 
the proof.
\end{proof}

\bigskip

\paragraph{\emph{Power-law trapping time}} 
For power-law distributed trapping times with exponent $q\leq 2$, the 
finite-time behavior of the MSD can be reasonably well fitted by a 
power-law. Using Proposition~\ref{prop:sigmarecurrence} we computed 
$\sigma_t^2$ for $1\leq t\leq 2^{17}$ and fitted the behavior to a 
power-law, $\sigma_t^2={\mathcal O}(t^\beta)$, with an exponent 
$\beta_N(q)$, which deviates from the expected exponent $\beta(q)=q-1$. 
We plot this effective exponent in Figure~\ref{fig:exponent<1} for time 
intervals $10\leq t\leq N$, with $N=2^{13}, 2^{15}$ and $2^{17}$. 
For each $N$, the behavior of $\beta_N(q)$ can be very well fitted by a 
sigmoidal function $\beta(q)=c+2(1-c)/(1+e^{r |q-3|^\eta})$.

\begin{center}
\begin{figure}[h]
\begin{tikzpicture}
\definecolor{mybarcolor}{RGB}{210 207 155}
\pgfmathsetmacro{\xmin}{1}
\pgfmathsetmacro{\ymin}{0.0}
\begin{axis}[ymin=0,%
  xmin=\xmin,%
  ymin=\ymin,
  ylabel=\textbf{$\hskip 60pt \beta_{N}(q)$\hskip 10pt $\beta(q)$},%
  xlabel=\textbf{\hskip 60pt  $q$},%
  grid=major,%
]
\addplot[ color = blue] table {%
   1.010000000000000   0.515563778388733
   1.060000000000000   0.512684815141280
   1.110000000000000   0.514201542044770
   1.160000000000000   0.519884106794074
   1.210000000000000   0.529407210091236
   1.260000000000000   0.542400852881217
   1.310000000000000   0.558492970801179
   1.360000000000000   0.577336533136565
   1.410000000000000   0.598619719823978
   1.460000000000000   0.622061652722420
   1.510000000000000   0.647397541641613
   1.560000000000000   0.674356979130971
   1.610000000000000   0.702638692121427
   1.660000000000000   0.731885213464733
   1.710000000000000   0.761662014622240
   1.760000000000000   0.791447308790266
   1.810000000000000   0.820639835289992
   1.860000000000000   0.848590622259795
   1.910000000000000   0.874659358609408
   1.960000000000000   0.898286915540927
   2.010000000000000   0.919066413639457
};
\addplot[ color = blue] table {%
   1.010000000000000   0.387108151223422
   1.060000000000000   0.388224729437350
   1.110000000000000   0.394115267559683
   1.160000000000000   0.404580690604252
   1.210000000000000   0.419334839177768
   1.260000000000000   0.438035656241604
   1.310000000000000   0.460312011926043
   1.360000000000000   0.485781450411924
   1.410000000000000   0.514057371776389
   1.460000000000000   0.544746663927528
   1.510000000000000   0.577440258494053
   1.560000000000000   0.611699798422074
   1.610000000000000   0.647044045751399
   1.660000000000000   0.682939110369300
   1.710000000000000   0.718796946834740
   1.760000000000000   0.753986342018275
   1.810000000000000   0.787859084873497
   1.860000000000000   0.819790711996817
   1.910000000000000   0.849230568052075
   1.960000000000000   0.875751501581279
   2.010000000000000   0.899087598039155
};
\addplot[ color = blue] table {%
   1.010000000000000   0.337394719514335
   1.060000000000000   0.340157456596949
   1.110000000000000   0.347906616497557
   1.160000000000000   0.360465337289053
   1.210000000000000   0.377568569558786
   1.260000000000000   0.398878392495984
   1.310000000000000   0.423998606308125
   1.360000000000000   0.452486505604784
   1.410000000000000   0.483861016079263
   1.460000000000000   0.517607539837315
   1.510000000000000   0.553180739567543
   1.560000000000000   0.590007081522237
   1.610000000000000   0.627489263273501
   1.660000000000000   0.665014624807474
   1.710000000000000   0.701969161268897
   1.760000000000000   0.737757714879474
   1.810000000000000   0.771829364004380
   1.860000000000000   0.803705254792431
   1.910000000000000   0.833004690829342
   1.960000000000000   0.859464818184449
   2.010000000000000   0.882950070940828
};
\addplot[domain=\xmin:2,samples=100, color=blue, dashed] {x-1};
\end{axis}
\end{tikzpicture}
\caption{
The curves $q\mapsto \beta_N(q)$ correspond to the exponent 
of the approximated power-law behavior of the MSD as a function 
of the trapping time distribution's exponent. We show these 
curves for total observation times $N=2^{13}, 2^{15}$ and $2^{17}$. 
The curves approach, as $N\to\infty$, the asymptotic 
exponent $q\mapsto\beta(q)=q-1$ (dashed line).} \label{fig:exponent<1}

\end{figure}
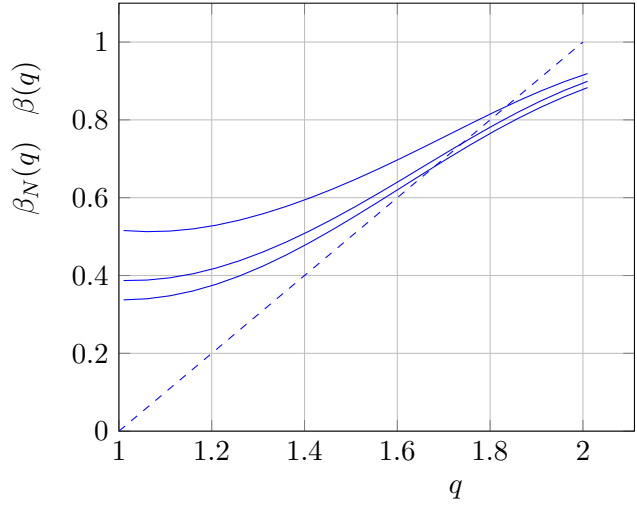
\end{center}

\bigskip
\subsection{Sub-diffusion via $X_t$}\label{subsec:subnormalX_t}
In this subsection we will consider trapping time probability 
distributions satisfying $\mathbb{E}(T^\alpha)<\infty$ and 
$\mathbb{E}(T^\beta)=\infty$, for some $\alpha\in (0,1)$ and all 
$\beta >\alpha$. This case is consistent with a sub-linear growth 
of the MSD dominated by a power-law of order $\alpha$. 
We will show that the expected number $N_t$ of steps the random 
walker makes up to time $t$ grows like $t^{\alpha}$. 

\bigskip
\begin{proposition}[$N_t$ concentration for 
$\alpha < 1$]~\label{prop:ConcentrationSubdiffusion}
Let us assume that $\mathbb{E}(T^\alpha)<\infty$ and 
$\mathbb{E}(T^\beta)=\infty$, for some $\alpha\in (0,1)$ and all 
$\beta >\alpha$. Then there exists a function $t\mapsto h(t)$ 
varying slower than any power-law and converging to $0$, and a 
constant $C>0$ such that   
\[
\mathbb{P}\left(N_t\notin t^\alpha\,\left[h(t),h^{-2}(t)\right]
         \right)\leq C\,h(t)
\]
for each $t\in\mathbb{N}$.
\end{proposition} 

\begin{proof}
We will start by proving that 
$\mathbb{P}(T\geq \tau)=h(\tau)\,(\tau+1)^{-\alpha}$ for some function 
$\tau\mapsto h(\tau)$ varying slower than any power-law and converging 
to zero. For this we use the inequalities,
\[
\frac{\mathbb{E}\left(T^{\beta}\right)-1}{\beta}\leq
\sum_{\tau' \leq \tau} 
(\tau'+1)^{\beta-1}\,\mathbb{P}(T\geq\tau')\leq  
\frac{\mathbb{E}\left((T+1)^{\beta}\right)}{\alpha},
\]
which are valid for all $\beta \leq 1$ and $\tau\in\mathbb{N}$, and 
are easily verified. With this we have
\[\mathbb{P}(T\geq \tau)\,(\tau+1)^{\alpha} \leq 
\sum_{0\leq \tau' \leq \tau} 
(\tau'+1)^{\alpha-1}\,\mathbb{P}(T\geq\tau')\leq  
\frac{\mathbb{E}\left((T+1)^{\alpha}\right)}{\alpha},\]
for all $\tau\in\mathbb{N}$. Hence, 
$\limsup_{\tau\to\infty}
\mathbb{P}(T\geq \tau)\,(\tau+1)^{\alpha}\leq \mathbb{E}((T+1)^\alpha)/\alpha$,
from which it readily follows that
$\lim_{\tau\to\infty}
\mathbb{P}(T\geq \tau)\,(\tau+1)^{\beta}=0$ for 
each $\beta<\alpha$.

\medskip \noindent 
Suppose now that 
$\limsup_{\tau\to\infty}(\tau+1)^{\beta}\mathbb{P}(T\geq \tau)<\infty$ for
some $\beta >\alpha$. In this case we have
$(\tau+1)^{\beta}\mathbb{P}(T\geq \tau) \leq M$ for all 
$\tau\in\mathbb{N}$ and some $M > 0$, from which it follows that
\[
\sum_{\tau \geq 0}(\tau+1)^{(\beta+\alpha)/2-1}\,\mathbb{P}(T\geq\tau)
\leq
\sum_{\tau \geq 0}\frac{M}{(\tau+1)^{1+(\beta-\alpha)/2}}
\leq M\frac{2+\beta-\alpha}{\beta-\alpha}<\infty.
\]
Hence $\mathbb{E}\left(T^{(\beta+\alpha)/2}\right)<\infty$ which contradicts
the fact that $\mathbb{E}\left(T^{\beta}\right)<\infty$ for all 
$\beta>\alpha$. Therefore $\lim_{\tau\to\infty}
(\tau+1)^{\beta}\mathbb{P}(T\geq \tau)=\infty$ for each $\beta >\alpha$.
This completes the proof that 
$\tau\mapsto h(\tau):=\mathbb{P}(T\geq\tau)/(\tau+1)^\alpha$ varies slower 
than any power-law. The fact that $\mathbb{E}(T^\alpha)<\infty$ ensures that
$h(\tau)$ converges to zero.

\medskip\noindent 
Now, for $n \geq t^\alpha\,h(t)^{-2}$ we have
\begin{align*}
\mathbb{P}\left(N_t\geq n\right)
&=\mathbb{P}\left(\sum_{1\leq k\leq n+1} T_k\leq t-n
           \right)
\leq \left(1-\mathbb{P}(T\leq t-n)\right)^{n}\leq 
     \left(1-\mathbb{P}(T\leq t)\right)^{n}\\
&=\left(1-h(t)\,(t+1)^{-\alpha}\right)^{n}
=\left(1-\frac{h(t)}{(t+1)^\alpha}\right)^{t^\alpha\,h(t)^{-2}}\\
&\leq 
 \exp\left(-\frac{t^{\alpha}}{(t+1)^\alpha\,h(t)}\right)
\leq\left(-\frac{t^{\alpha}}{(t+1)^\alpha\,h(t)}\right)
\leq \frac{h(t)}{(1+t^{-1})^\alpha},
\end{align*}
provided $h(t)\leq (1+t^{-1})^\alpha$. By taking 
$t_0:=\min\{t\geq (2^{1/\alpha}-1)^{-1}:\ 
h(\tau )\leq (1+\tau^{-1})^\alpha\,\, \forall \tau\geq t\}$
and $h_0=\min_{0\leq t\leq t_0}\,h(t)$
we have 
$\mathbb{P}(N_t\geq t^\alpha\,h(t)^{-2})
\leq \max(h_0^{-1},2)\, h(t)$ for all $t\in \mathbb{N}$.

\medskip \noindent 
For the other inequality, suppose that $n-1\leq t^\alpha\,h(t)\leq n$. 
Using Eq.~\ref{eq:NumberOfHits} and Bahr-Esseen's inequality we have
\[
\mathbb{P}\left(N_t\leq n\right)
=\mathbb{P}\left(\sum_{1\leq k\leq n} T_k\geq t-n
           \right)
\leq \frac{n\,\mathbb{E}(T^{\alpha})
         }{(t-n)^\alpha}
\leq \frac{\mathbb{E}(T^{\alpha})
       }{g(t)\,
  \left(1-t^{\alpha-1}\,h(t)-t^{-1}\right)^{\alpha}}.
\]
Let 
$t_1=\min\{t\in\mathbb{N}:\,
1-\tau^{\alpha-1}\,h(\tau)-\tau^{-1} < 2^{-1/\alpha}\,
\, \forall \, \tau\geq t\}$ and 
$h_1=\min_{0\leq t\leq t_1}\,h(t)$, then 
$\mathbb{P}(N_t\leq t^\alpha\,h(t))\leq 
\max(h_1^{-1},2\,\mathbb{E}(T^\alpha))\, h(t)$ for all 
$t\in \mathbb{N}$. 

\medskip\noindent
The proof is done by taking 
$C:=\max(h_0^{-1},h_1^{-1},2,2\mathbb{E}(T^\alpha))$.
\end{proof}

\medskip\begin{remark} According to Gnedenko's Theorem, for 
$\mathbb{P}\left(T\geq t\right)=(t+1)^{-\alpha}\,h(t)$ with 
$t\mapsto h(t)$ a slowly varying function, the sequence of 
distributions 
$x\mapsto\mathbb{P}\left(n^{-1/\alpha}\, \sum_{k=1}^n T_k\leq x\right)$ 
converges in law to $x\mapsto\int_{-\infty}^x F_\alpha(z)\, dz$, 
where $F_\alpha$ is an $\alpha$-stable distribution with support in 
$\mathbb{R}^+$. This result and Barry-Esseen Inequality allow us 
to derive an analogous to the Central Limit Theorem. We have the 
following.
\end{remark}

\medskip
\begin{theorem}~\label{teo:CLTnonDiffusif}
Let us assume that $\mathbb{P}\left(T\geq t\right)=(t+1)^{-\alpha}\,h(t)$ 
with $t\mapsto h(t)$ a slowly varying function. Then
\[
\lim_{t\to\infty}
\mathbb{P}\left(\frac{X_t}{t^{\alpha/2}}\in I\right)=
\int_{\mathbb{R}^+}F_\alpha(z)\,dz
     \frac{1}{\sqrt{2\pi\,z}}
          \int_{I}\exp\left(-\frac{x^2}{2z}\right)\,dx.
\] 
\end{theorem}

\medskip\begin{remark}
Under the hypotheses of this theorem, $\mathbb{E}(T^\gamma)<\infty$ and 
$\mathbb{E}(T^\beta)=\infty$, for all $\gamma <\alpha<\beta$. In this case 
it is not possible, in general, to explicitly determine $F_\alpha$ or the 
speed of convergence, but still some features of the limit behavior 
can be derived from this expression, for instance the fact that it 
is a symmetric distribution and that all absolute moments of order 
larger than $\alpha$ diverge. The scaling indicates that the support 
of the measure is concentrated around $t^\alpha$, but this was already 
clear from Proposition~\ref{prop:ConcentrationSubdiffusion}. 
\end{remark}

\begin{proof}
For each $\delta>0$ and $w\in \mathbb{R}^+$ we have  
\begin{align*}
\mathbb{P}\left(w\leq\frac{N_t}{t^\alpha}\leq w+\delta\right)
&=\mathbb{P}\left(
\sum_{k=1}^{[t^\alpha\, w]}T_k\leq t-t^\alpha\,w
\right)-
\mathbb{P}
\left(
\sum_{k=1}^{[t^\alpha\, (w+\delta)]+1}T_k\leq t-t^\alpha\,(w+\delta)
\right)\\
&=\mathbb{P}\left(\frac{
\sum_{k=1}^{[t^\alpha\, w]}T_k}{[t^\alpha\, w]^{1/\alpha}}
\leq \frac{t-t^\alpha\,w}{[t^\alpha\, w]^{1/\alpha}}
\right)\\
& \hskip 55pt -
\mathbb{P}
\left(
\frac{\sum_{k=1}^{[t^\alpha\, (w+\Delta)]+1}T_k
    }{([t^\alpha\, (w+\delta)]+1)^{1/\alpha}}
\leq \frac{t-t^\alpha\,(w+\delta)
        }{([t^\alpha\, (w+\delta)]+1)^{1/\alpha}}
\right).
\end{align*}
Hence, by virtue of Theorems 4.2.1 and 4.2.2 in~\cite{1971Ibragimov&Linnik}, 
we obtain
\[
\left|\mathbb{P}\left(N_t/t^\alpha\in [w,w+\delta]\right)-
\int_{(w+\delta)^{-1/\alpha}}^{w^{-1/\alpha}}F_\alpha(z)\,dz\right|
\leq g_1(t\,w^{1/\alpha}),\]
where $s\mapsto g_1(s)>0$ is monotonously decreasing, such that 
$\lim_{s\to\infty}g_1(s)=0$. On the other hand, for $m>n$ we have
\[\left|
\int_{\sqrt{t^{\alpha}/n}\, I}e^{-x^2/2}dx-
\int_{\sqrt{t^{\alpha}/m}\, I}e^{-x^2/2}dx
\right|\leq e^{-1/2}\left(\sqrt{\frac{m}{n}}-1\right).
\]
Hence, taking into account Equation~\eqref{eq:NumberOfHits}, Barry-Esseen's inequality and Proposition~\ref{prop:ConcentrationSubdiffusion}, we obtain
\[
\left|\mathbb{P}
\left(\frac{X_t}{t^{\alpha/2}}\in I\right)
-\frac{1}{\sqrt{2\pi}}\sum_{k=1}^{\lceil (h^{-2}(t)-h(t))/\delta\rceil}
\int_{(h(t)+k\delta)^{-1/\alpha}}^{(h(t)+(k-1)\delta)^{-1/\alpha}}
F_\alpha(z)\,\int_{I/\sqrt{z}}e^{-x^2/2}\,dx
\,dz\right|\leq g_2(t,\delta),
\]
with $g_2(t,\delta)=g_1(t\,h(t)^{1/\alpha})+
2/\sqrt{2\pi\,t^{\alpha}h(t)}+e^{-1/2}(\sqrt{(h(t)+\delta)/h(t)}-1)$.
Finally, by taking $\delta=\delta(t):=h^2(t)$ we obtain
\[
\left|\mathbb{P}
\left(\frac{X_t}{t^{\alpha/2}}\in I\right)
-\int_{h^{2/\alpha}(t)}^{h(t)^{-1/\alpha}}
F_\alpha(z)\frac{1}{\sqrt{2\pi\,z}}
\int_{I}\exp\left(\frac{-x^2}{2z}\right)\,dx
\,dz\right|\leq g_3(t),
\]
with $g_3(t)=g_2(t,h^2(t))+h(t)^{2/\alpha+4}/\alpha$, and the result follows.
\end{proof}

\bigskip
\section{Summary and final remarks}

\subsection{Summary}
\begin{itemize} 
\item[A.] 
Thanks to the recurrence relation established in 
Proposition~\ref{prop:sigmarecurrence}, we were able to characterize 
the dynamical regimes of the trapped random walk. According to 
Proposition~\ref{prop:unboundedness}, regardless of the trapping time 
distribution, the MSD of the walk will always diverge. Hence, 
dynamical confinement is impossible in this model. Depending on 
the characteristics of the trapping time distribution, the random walk can 
display a diffusive or sub-diffusive dynamics. 

\item[B.]
The diffusive behavior, which takes place when the mean trapping time 
is finite, can be observed at finite time if the second moment of the 
trapping time is finite, as established in Theorem~\ref{teo:NormalMSD}. 
On the contrary, for trapping time distributions with finite mean and 
diverging second moment, the diffusive behavior takes place asymptotically, 
and for any finite time we can measure an effective sub-diffusive behavior.
To illustrate this sub-diffusive behavior we have computed the MSD as 
a function of time using a collection of power-law trapping time 
distributions (see Figure~\ref{fig:exponent>1}). 

\item[C.]
We consider a class of trapping time distributions consisting of a 
leading power-law behavior with fluctuations around this leading 
behavior, varying slower than any power law. As established in 
Theorem~\ref{teo:MSD-Subdiffusion}, in this case the MSD of the walk 
grows following a power-law directly related to the leading term 
of the trapping time distribution. This power-law behavior holds 
asymptotically, but at any finite time a deviation of smaller order 
can be observed. We illustrate this deviation with numerical computations 
from a collection of power-law distributions 
(see Figure~\ref{fig:exponent<1}).
\end{itemize}
In the next table we summarize the behavior of the MSD 
as a function of the trapping time distribution, for power-law trapping 
time distributions of the type $p(\tau)=\tau^{-q}/\zeta(q)$.

\medskip
\begin{center}
\begin{tabular}{|c|c|c|}
\hline
Exponent       &  MSD leading term &  Finite time deviation \\
\hline
$q > 3$        & $\mathcal{D}\,t$  & $\mathcal{O}(1)$\\
\hline
$2 < q \leq 3$ & $\mathcal{D}\,t$  & $\mathcal{O}(t^{3-q})$\\
\hline
$1 < q \leq 2$ & $t^{q-1}$         & $h(t)\,t^{q-1}$\\      
\hline
\end{tabular}
\end{center}
Here $t\mapsto h(t)$ is a function varying slower than any power-law and 
converging to 0.

\medskip\noindent
In Figure~\ref{fig:exponent} we depict the exponent of the approximated 
power-law behavior for different finite observation times for the same 
family of trapping time distributions. 

\begin{center}
\begin{figure}[h]
\begin{tikzpicture}
\definecolor{mybarcolor}{RGB}{210 207 155}
\pgfmathsetmacro{\xmin}{0.9}
\pgfmathsetmacro{\ymin}{-0.1}
\begin{axis}[%
  xmin=\xmin,%
  ymin=\ymin,
  ylabel=\textbf{$\hskip 60pt \beta_{N}(q)$\hskip 10pt $\beta(q)$},%
  xlabel=\textbf{\hskip 60pt  $q$},%
  grid=major,%
]
\addplot[ color = blue]  table {%
   1.010000000000000   0.515563778388733
   1.160000000000000   0.519884106794074
   1.310000000000000   0.558492970801179
   1.460000000000000   0.622061652722420
   1.610000000000000   0.702638692121427
   1.760000000000000   0.791447308790266
   1.910000000000000   0.874659358609408
   2.060000000000000   0.936791849435148
   2.210000000000000   0.972575043304649
   2.360000000000000   0.989115592568454
   2.510000000000000   0.995811028876723
   2.660000000000000   0.998401015742535
   2.810000000000000   0.999422598780135
   2.960000000000000   0.999850600317113
};
\addplot[ color = blue] table{%
   1.010000000000000   0.387108151223422
   1.160000000000000   0.404580690604252
   1.310000000000000   0.460312011926043
   1.460000000000000   0.544746663927528
   1.610000000000000   0.647044045751399
   1.760000000000000   0.753986342018275
   1.910000000000000   0.849230568052075
   2.060000000000000   0.9191505084461520
   2.210000000000000   0.961155241420844
   2.360000000000000   0.982574258803797
   2.510000000000000   0.992450890590414
   2.660000000000000   0.996837030720230
   2.810000000000000   0.998805007705768
   2.960000000000000   0.999724886829804
};
\addplot[ color = blue] table{%
   1.010000000000000   0.337394719514335
   1.160000000000000   0.360465337289053
   1.310000000000000   0.423998606308125
   1.460000000000000   0.517607539837315
   1.610000000000000   0.627489263273501
   1.760000000000000   0.737757714879474
   1.910000000000000   0.833004690829342
   2.060000000000000   0.903449532142498
   2.210000000000000   0.948443343357862
   2.360000000000000   0.974004903350668
   2.510000000000000   0.987475360812692
   2.660000000000000   0.994339295836294
   2.810000000000000   0.997834788345466
   2.960000000000000   0.999654396883074
};
\addplot[domain=1:2,samples=100, color=blue, dashed] {x-1};
\addplot[domain=2:3,samples=100, color=blue, dashed]{1};
\end{axis}
\end{tikzpicture}
\caption{The curves $q\mapsto\beta_N(q)$ correspond to the exponent 
of the approximated power-law behavior of MSD as a function of the 
trapping time distribution's exponent. We show these curves for total 
observation times $N=2^{13},2^{15}$ and $2^{17}$. The curves approach, 
as $N\to\infty$, the asymptotic exponent $q\mapsto\beta(q)=\min(1,q-1)$ 
(dashed line).}
\end{figure}
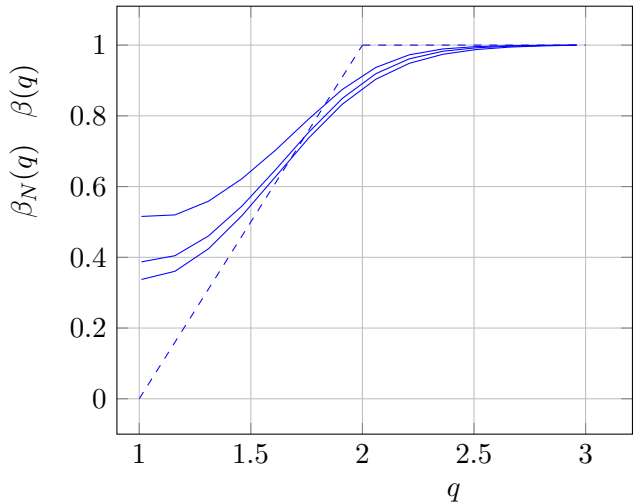\label{fig:exponent}
\end{center}

\bigskip\noindent

\begin{itemize}
\item[D.] If the mean trapping time is finite, then the trapped 
random walk satisfies a Central Limit Theorem. The normalization 
required as well as the speed of convergence towards the normal 
distribution both depend on the characteristics of the trapping 
time distribution. In Theorem~\ref{teo:CentralLimitTheorem} we treat 
the case where the trapping time has a finite fractional moment 
strictly larger than one, for which we prove a power-law convergence 
towards the limit normal distribution. In the case of a trapping 
time distribution belonging to the domain of attraction of the 
Cauchy distribution, Theorem~\ref{teo:CentralLimitTheoremOne}, the 
convergence we found towards the limit normal law is slower than 
any power-law. 

\item[E.] Finally, when the mean trapping time diverges, any trapping 
time distribution belonging to the domain of attraction of a stable 
law leads to a limiting distribution for the trapped random walk, once 
the length is properly re-normalized. The limit distribution is a convex 
combination, governed by the corresponding stable law, of normal 
distributions (see Theorem~\ref{teo:CLTnonDiffusif}).
\end{itemize}

\medskip\noindent
In the next table we present the behavior of the required normalization 
and the speed of convergence towards the normal distribution as a function 
of the trapping time leading exponent, for trapping time distributions 
of the type $p(\tau)=\tau^{-q}/\zeta(q)$.

\medskip
\begin{center}
\begin{tabular}{|c|c|c|}
\hline
Exponent       &  Normalization &  Speed of convergence \\
\hline
$q > 2$        & $(t/\mu)^{1/2}$  & 
                  $\mathcal{O}\left(t^{-q/(2+q)}\right)$\\
\hline
$1 < q \leq 2$    & $t^{(q-1)/2}$  & Slower than any power-law.
\\      
\hline
\end{tabular}
\end{center}

\medskip\noindent 
In Figure~\ref{fig:z-tPlane} we depict several possible trajectories of 
the random walk in the diffusive regime. In the figure we also depict 
the behavior of the standard deviation $\sigma_t\propto \sqrt{t}$. According 
to Theorems~\ref{teo:CentralLimitTheorem} 
and~\ref{teo:CentralLimitTheoremOne}, the random walk spreads
around the vertical line $X_t=0$, with oscillations 
contained inside the curves $t\mapsto\pm\sigma_t$.
The distribution of the number of steps $N_t$, the random walker makes 
until time $t$, is concentrated between 
$n_1(t)\approx (t/\mu)(1-1/d(t))$ and 
$n_2(t)\approx (t/\mu)(1+1/d(t))$. Here $t\mapsto d(t)$ is a
diverging function which can be either a sub-linear power-law 
or a function varying slower than any power-law. The first case
take place when $\mathbb{E}(T^\alpha)<\infty$ for some 
$\alpha > 1$. For a power-law trapping time distribution
$p(\tau)=(\tau+1)^{-q}/\zeta(q)$, this happens when $q > \alpha+1$. 
On the other hand, $t\mapsto d(t)$ varies slower than any power-law 
when $\mathbb{E}(T^\alpha)=\infty$ for each $\alpha>1$. This is
the case of a trapping time distribution 
$p(\tau)=(\tau+1)^{-q}/\zeta(q)$ with 
$q=1+\inf\{\alpha>0:\, \mathbb{E}(T^\alpha)=\infty\}$. 
As we have seen, the distributuion of $N_t$ determines the finite time 
spread of the trapped random walk and their scaling properties.
The image in the case of the sub-diffusive regime 
$\sigma_t\sim t^{\alpha/2}$, is qualitatively the same. In this case, 
the spread of the random walk is bounded, with high probability, between 
the curves $h(t)\,t^\alpha$ and $h^{-2}(t)\, t^\alpha$, where $h(t)$ 
is a positive function converging to zero slower than any power law, 
which can be computed from the trapping time distribution.

\bigskip
\begin{center}
\begin{figure}[h]
\begin{tikzpicture}[scale=0.55]
\draw[help lines] (1,0) grid (15,15);
\draw[<->] (0,0) -- (16,0);
\draw[->] (8,0) -- (8,16);
\draw[ ->] (8,0) to [out=45, in=250] (15,15);
\draw[ ->] (8,0) to [out=135, in=290] (1,15);
\node [right] at (8,15.5) {$t$};  
\node [below] at (15,0) {$X_t,\ \sigma_t$};   
\draw [->, thick, brown] 
(8,0) -- (8,1) -- (7,2) -- (7,3) -- (7,4) -- 
(6,5) -- (6,6) -- (7,7) -- (7,8) -- (7,9) -- (6,10) -- (7,11) --
(6,12) -- (6,13) -- (7,14) -- (7,15) ;    
\draw [->, thick, red] 
(8,0) -- (8,1) -- (8,2) -- (9,3) -- (9,4) -- 
(9,5) -- (9,6) -- (10,7) -- (9,8) -- (8,9) -- (8,10) -- (8,11) --
(7,12) -- (7,13) -- (6,14) -- (6,15) ;  
\draw [->, thick, green] 
(8,0) -- (7,1) -- (7,2) -- (7,3) -- (6,4) -- 
(5,5) -- (6,6) -- (6,7) -- (6,8) -- (5,9) -- (5,10) -- (6,11) --
(6,12) -- (7,13) -- (7,14) -- (8,15) ;
\draw [->, thick, blue] 
(8,0) -- (9,1) -- (9,2) -- (9,3) -- (10,4) -- 
(9,5) -- (9,6) -- (8,7) -- (9,8) -- (8,9) -- (8,10) -- (7,11) --
(7,12) -- (7,13) -- (7,14) -- (8,15) ;
\draw [->, thick, yellow] 
(8,0) -- (7,1) -- (7,2) -- (7,3) -- (7,4) -- 
(8,5) -- (8,6) -- (8,7) -- (9,8) -- (8,9) -- (9,10) -- (9,11) --
(10,12) -- (11,13) -- (10,14) -- (10,15) ;
\node at (13.7,9.4) {$\sigma_t$};  
\node at (8.5,11.3) {$X_t$}; 
\end{tikzpicture}
\caption{
Five possible trajectories for the random walk. The spread 
of the random walk around the vertical line $X_t=0$ 
is of order $\sigma_t$ whereas for any finite time this spreading fluctuates between the concentration bounds $\sqrt{n_1(t)}$ and 
$\sqrt{n_2(t)}$.
}\label{fig:z-tPlane}
\end{figure}
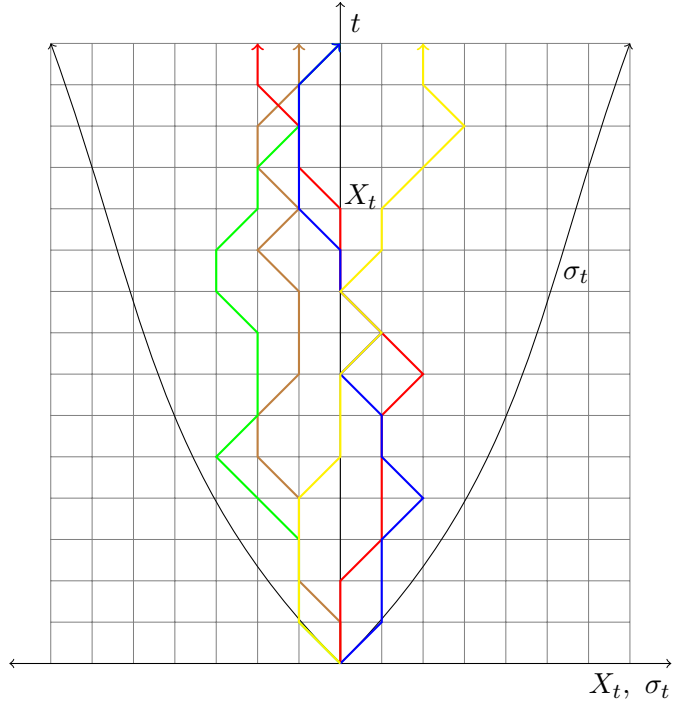
\end{center}

\subsection{Final Remarks}
\begin{itemize}
\item[A.] In the case of finite mean trapping time we have two scenarios: 
either there exists $\alpha >1$ such that $\mathbb{E}(T^\alpha)<\infty$, or 
$\mathbb{E}(T^\alpha)=\infty$ for each $\alpha >1$. In the first case the 
behavior of the system is controlled by the exponent
\[
\gamma=\sup\{\alpha > 1:\ \mathbb{E}(T)< \infty\}
      =\inf\{\alpha > 1:\ \mathbb{E}(T) = \infty\}.
\] 
In this case $\sqrt{\mu/t}\,X_t$ converges in law to the normal 
distribution, with a speed of the order $(1-\gamma)/(1+\gamma)$. This 
fact can be easily derived from Theorem~\ref{teo:CentralLimitTheorem}, 
taking into account that all constants involved in upper bounds vary 
continuously with the exponents. In the case where 
$\mathbb{E}(T^\alpha)=\infty$ for each $\alpha >1$, then in order to 
ensure convergence to the Gaussian, we have to restrict ourselves to 
trapping times in the domain of attraction of Cauchy's distribution 
(Theorem~\ref{teo:CentralLimitTheoremOne}). The same kind of restriction 
has to be assumed when dealing with distributions with infinite mean 
trapping time (Theorem~\ref{teo:CLTnonDiffusif}). It remains to 
determine whether these restrictions are of an essential nature or can 
be weakened and replaced by a hypothesis concerning only the value of 
the exponent of the largest finite moment as in 
Theorem~\ref{teo:CentralLimitTheorem}. 

\item[B.] In Theorem~\ref{teo:MSD-Subdiffusion} and 
Proposition~\ref{prop:ConcentrationSubdiffusion} we consider trapping time 
distributions with a leading power-law behavior, modified by a term varying 
slower than any power-law. This is required to ensure the 
power-law growth of the MSD. Nevertheless, the convergence towards a 
limit law for the corresponding random walk requires a little more 
control on the decay of the distribution: to insure convergence, 
the distribution has to be regularly varying, which is a strictly 
stronger condition. In this case, the condition cannot be weakened 
since regular variation is a necessary condition for being in the 
domain of attraction of a stable law.
\end{itemize}

\bigskip

\section*{Acknowledgments}

\noindent We thank CONACyT-ECOS (grant No. M16M01) and Fundaci\'on Marcos 
Moshinsky (C\'atedra de Investigaci\'on para J\'ovenes Cient\'\i ficos 2016) 
for their financial support. EU thanks Gelasio Salazar for his careful 
reading and valuable suggestions. 

\bigskip
\bibliographystyle{plain}

\appendix
\section{Variation slower than any power-law}\label{app:slower}

\noindent 
A strictly positive function $g:\mathbb{N}\to (0,\infty)$ varies slower 
than any power-law if $\lim_{t\to\infty}g(t)\,t^{-\epsilon}=0$ and 
$\lim_{t\to\infty}g(t)\,t^{\epsilon}=\infty$ for any $\epsilon>0$. 
In particular, any regularly varying function of order zero varies 
slower than any power-law. We have the following.

\medskip\noindent {\bf Claim 1.}
\emph{Let $t\mapsto g(t)$ and $t\mapsto h(t)$ be two functions varying 
slower than any power-law. Then the following are functions varying 
slower than any power-law.
\begin{center}
\begin{tabular}{lll}
a) $t\mapsto \lambda\,g(t)$ with $\lambda >0$, 
   & b) $t\mapsto g(t)+h(t)$, 
   & c) $t\mapsto g(t)\,h(t)$, \\  
d) $t\mapsto 1/g(t)$,       
   & e) $t\mapsto \min_{\mu\,t\leq \tau\leq \lambda\, t}g(\tau)$,  
   & f) $t\mapsto \max_{\mu\,t\leq \tau\leq \lambda\, t}g(\tau)$
                            with $0\leq \mu < \lambda$.
\end{tabular}
\end{center}
Furthermore, if $g(t)\leq f(t)\leq h(t)$ for each $t\in\mathbb{N}$, then \\
g) $t\mapsto f(t)$ varies slower than any power-law, \\
and, if $\tau\mapsto P(\tau)\geq 0$ is such that
$0 < \sum_{s \geq 1}P(s)\,(s+1)^{\epsilon_0} < \infty$ for some 
$\epsilon_0 >0$, then \\
h) $t\mapsto \sum_{s\geq 1} P(s)\,\max_{s\,t < \tau\leq (s+1)\,t}g(\tau)$, 
varies slower than any power-law as well.}

\medskip \begin{proof}
Items $a)$ to $d)$ are easily proved and are let to the reader. 

\medskip \noindent For $e)$ let 
$\ell_t=\max\left\{\mu\,t\leq \tau\leq \lambda\,t:\
           g(\tau)=\min_{\mu\,t\leq s\leq \lambda\, t}g(s)\right\}$.
Supposing that $\lim_{t\to\infty}\ell_t=\infty$, we have
\begin{align*}
\lim_{t\to\infty}t^{-\epsilon}\,\min_{\mu\,t\leq \tau\leq \lambda\, t}g(\tau)
& = \lim_{t\to\infty}t^{-\epsilon}\,g(\ell_t)\leq 
     \lambda^{\epsilon}\,\lim_{t\to\infty}\ell_t^{-\epsilon}\,g(\ell_t)
  = 0,\\
\lim_{t\to\infty}t^\epsilon\,\min_{\mu\,t\leq \tau\leq \lambda\, t}g(\tau)
& = \lim_{t\to\infty}t^\epsilon\,g(\ell_t)\geq 
    \lambda^{-\epsilon}\,\lim_{t\to\infty}\ell_t^\epsilon\,g(\ell_t)
  =\infty.\\
\end{align*}
Now, if $\lim_{t\to\infty}\ell_t=\ell<\infty$, in which case $\mu=0$, then $g(\ell)\leq \min_{\tau\leq \lambda\,t}g(t)\leq \max_{\tau\leq\ell}g(\tau)$ for each $t\in\mathbb{N}$, and therefore $t\mapsto \min_{\tau\leq \lambda\,t}g(t)$ varies slower than any power-law.  

\medskip\noindent
Item $f)$ directly follows from $e)$ and $d)$ by noticing that $\max_{\mu\,t\leq t\leq\lambda\,t}g(\tau)=\left(\min_{\mu\,t\leq \tau\leq \lambda\,t}1/g(\tau)\right)^{-1}$.

\medskip\noindent For $g)$, it is enough to notice that
\begin{align*}
\lim_{t\to\infty}t^{-\epsilon}\,f(t)
&\leq \lim_{t\to\infty}t^{-\epsilon}\,h(t)=0,\\
\lim_{t\to\infty}t^{\epsilon}\,f(t)
&\geq \lim_{t\to\infty}t^{\epsilon}\,g(t)=\infty.
\end{align*}
for each $\epsilon >0$.

\medskip\noindent 
For $h)$, let $u_{s,t}=\max\left\{s\,t < \tau\leq (s+1)\,t:\
      g(\tau)=\max_{\mu\,t\leq s\leq \lambda\, t}g(s)\right\}$.
For each $\epsilon \leq 2\,\epsilon_0$ we have
\begin{align*}
t^{-\epsilon}\,
 \sum_{s\geq 1}P(s)\max_{s\,t <\tau\leq (s+1)\,t}g(\tau)
&=
 t^{-\epsilon/2}\,\sum_{s\geq 1}P(s)\,g(u_{s,t})\,u_{s,t}^{-\epsilon/2}
 \left(\frac{t}{u_{s,t}}\right)^{-\epsilon/2}\\
&\leq 
 t^{-\epsilon/2}\,
 \sum_{s\geq 1}P(s)\,(s+1)^{\epsilon/2}\,g(u_{s,t})\,u_{s,t}^{-\epsilon/2},\\
t^{\epsilon}\,
 \sum_{s\geq 1}P(s)\max_{s\,t <\tau\leq (s+1)\,t}g(\tau)
&=
 t^{\epsilon/2}\,\sum_{s\geq 1}P(s)\,g(u_{s,t})\,u_{s,t}^{\epsilon/2}
 \left(\frac{t}{u_{s,t}}\right)^{\epsilon/2}\\
&\geq 
 t^{\epsilon/2}\,
 \sum_{s\geq 1}P(s)\,(s+1)^{-\epsilon/2}\,g(u_{s,t})\,u_{s,t}^{\epsilon/2}.
\end{align*}
Since $\lim_{t\to\infty}g(t)\,t^{-\epsilon/2}=0$ then 
$g(t)\,t^{-\epsilon/2}$ is bounded, and therefore
\[
\lim_{t\to\infty}t^{-\epsilon}\,
 \sum_{s\geq 1}P(s)\max_{s\,t <\tau\leq (s+1)\,t}g(\tau)
\leq 
 \left(\max_{s\geq 1}g(s)\,s^{-\epsilon/2}\right)\,
 \left(\sum_{s\geq 1}P(s)\,(s+1)^{\epsilon/2}\right)\,
 \lim_{t\to\infty} t^{-\epsilon/2}=0. 
\]
On the other hand, since $g(t)\,t^{\epsilon/2}$ diverges, then
\[
\lim_{t\to\infty}t^{\epsilon}\,
 \sum_{s\geq 1}P(s)\max_{s\,t <\tau\leq (s+1)\,t}g(\tau)
\geq 
 \left(\sum_{s\geq 1}P(s)\,(s+1)^{-\epsilon/2}\right)\,
 \lim_{t\to\infty} t^{\epsilon/2}=\infty. 
\]
The same obviously holds for $\epsilon > 2\epsilon_0$.
\end{proof}

\end{document}